\documentclass[10pt]{amsart}
\usepackage{graphicx} 
\usepackage{amsmath,amsfonts}
\usepackage{amssymb,amsbsy,mathrsfs,amscd,eucal,wrapfig}
\usepackage{color}
\usepackage{geometry}
\usepackage{mathtools}
\usepackage{amsfonts,bm}
\usepackage[english]{babel}

\usepackage{tikz}
\usetikzlibrary{arrows,shapes}
\usetikzlibrary{arrows, decorations.markings,fit}
\usetikzlibrary{calc,3d}
\usepackage{lipsum}  
\usepackage{tikz-3dplot}

\usepackage{pgfplots}

\usepackage[all,cmtip]{xy}

\usetikzlibrary{arrows,positioning} 
\tikzset{
    >=stealth',
    punkt/.style={
           rectangle,
           rounded corners,
           draw=black, very thick,
           text width=6.5em,
           minimum height=2em,
           text centered},
    pil/.style={
           ->,
           thick,
           shorten <=2pt,
           shorten >=2pt,}
}

\allowdisplaybreaks[4]

\newcommand{\bu}{\bm u}

\newcommand{\bv}{\bm v}
\newcommand{\bL}{\bm L}
\newcommand{\bbR}{\mathbb{R}}
\newcommand{\nab}{\nabla}

\newcommand{\p}{\partial}

\newcommand{\bH}{\bm H}

\newcommand{\bV}{\bm V}
\newcommand{\bW}{\bm W}
\newcommand{\calT}{\mathcal{T}}
\newcommand{\bX}{\bm X}
\newcommand{\bw}{\bm w}

\newcommand{\Div}{{\rm div}\,}
\newcommand{\mct}{\mathcal{T}_h}
\newcommand{\balpha}{{\bm \alpha}}
\newcommand{\bn}{{\bm n}}
\newcommand{\bpol}{\boldsymbol{\pol}}
\newcommand{\pol}{\EuScript{P}}

\newcommand{\bbeta}{{\bm \beta}}
\newcommand{\bt}{{\bm t}}
\newcommand{\bE}{{\bm E}}
\newcommand{\bI}{{\bm I}}
\newcommand{\bPi}{{\bm \Pi}}
\newcommand{\brho}{{\bm \rho}}
\newcommand{\calE}{\mathcal{E}}

\numberwithin{equation}{section}

\newtheorem{theorem}{Theorem}[section]
\newtheorem{lemma}[theorem]{Lemma}

\newtheorem{corollary}[theorem]{Corollary}

\theoremstyle{definition}
\newtheorem{definition}[theorem]{Definition}
\newtheorem{remark}[theorem]{Remark}

\title{Divergence--free Scott--Vogelius elements on curved domains}
\author[M. Neilan and M.B. Otus]{Michael Neilan\address{Department of Mathematics, University of Pittsburgh, Pittsburgh, PA 15260}\email{neilan@pitt.edu}
\and M. Baris Otus \address{Department of Mathematics, University of Pittsburgh, Pittsburgh, PA 15260}\email{mbo13@pitt.edu}}

\thanks{Supported in part by the National Science Foundation  grant DMS-2011733}

\begin{document}

\maketitle

\begin{abstract}
We construct and analyze an isoparametric
finite element pair for the Stokes
problem in two dimensions.  The pair is defined by mapping the Scott-Vogelius
finite element space via a Piola transform.
The velocity space has the same degrees of freedom as the quadratic Lagrange
finite element space, and therefore, 
the proposed spaces
reduce to the Scott-Vogelius pair in the interior of the domain.  We prove
that the resulting method converges with optimal order, is divergence--free, and is pressure robust.
Numerical examples are provided which support the theoretical results.
\end{abstract}

\section{Introduction}
Isoparametric finite element methods 
are a well-known and extensively studied
technique to approximate PDEs on smooth domains.
Such schemes 
use polynomial diffeomorphisms between
reference and physical elements with degree
dictated by the approximation properties of the underlying finite element space.
The use of such mappings yield curved elements
on the boundary that, while still do not conform exactly 
to the physical domain, generally lead to higher--order approximations
and mitigate the geometric error.
In particular, the resulting geometric error
is generally of the same order as the discretization error, and thus,
the resulting methods are potentially robust with respect to rates
of convergence.
The implementation and analysis of isoparametric
elements for second--order, scalar elliptic problems
are well--established, and classical theories
exist \cite{Zlamal73,CiarletRaviart72,Lenoir86,Brenner,RScottThesis73}.
On the other hand, isoparametric elements
for mixed problems, in particular the Stokes problem,
is less developed \cite{BaeKim04,Verfurth87,Dione15}.

In this paper, we adopt and expand the isoparametric
framework to construct a divergence--free method for incompressible 
flow, i.e., schemes that yield discrete velocity solutions that
are divergence--free pointwise.  The scheme is also
pressure-robust, i.e., the gradient part of the source function
only influences the discrete pressure solution.  This feature
allows a decoupling of errors between the velocity and pressure,
which is beneficial for situations with fluid flow with large pressure gradient and/or small viscosity.
Such divergence-free and pressure-robust finite element schemes seem to be gaining
in popularity \cite{JLMNR17,Zhang05ABC,GuzmanNeilan18,ArnoldQin92,LinkeCR1,KanSharma14,VEM17A}, although, as far as we are aware, the methods have only been constructed
on polytopal domains.  Thus, divergence--free methods are currently
limited to second--order accuracy (formally) on general domains with smooth boundary.

The basis of our construction is the lowest-order two-dimensional Scott-Vogelius pair
defined on Clough-Tocher refinements, i.e., simplicial triangulations
obtained by connecting the vertices of each triangle in a given mesh to its barycenter.
In this case, the velocity space is the space of continuous, piecewise
quadratic polynomials, and the pressure space is the space
of (discontinuous) piecewise linear polynomials.  It is known, on affine Clough-Tocher meshes, 
this pair is stable, and the corresponding scheme is divergence-free
and pressure-robust.  However, a direct application of 
the isoparametric paradigm to this pair leads to a method with neither of these desirable properties.
Indeed, the Scott-Vogelius pair, defined by standard isoparametric mappings, is given by
\begin{subequations}
\label{eqn:pairComp}
\begin{align}
\breve{\bV}_h & = \{\bv\in \bH^1_0(\Omega_h):\ \bv|_K  = \hat\bv \circ F^{-1}_K,\ \exists \hat\bv\in \bpol_2(\hat{T})\ \forall K\in \mct^{ct}\},\\
\breve{Q}_h & = \{q\in L^2_0(\Omega_h):\ q|_K = \hat q \circ F_K^{-1},\ \exists \hat q  \in \pol_{1}(\hat{T})\ \forall K\in \mct^{ct}\},
\end{align}
\end{subequations}
where $\hat T$ is a reference triangle, $\pol_k(\hat T)$ denotes the space 
of polynomials of degree $\le k$ on $\hat T$,  $F_K:\hat T\to K$ is a quadratic diffeomorphism,
and $\mct^{ct}$ is the Clough-Tocher refinement of a simplicial triangulation $\mct$
(cf.~Section \ref{sec:Prelim} for a detailed explanation of the notation).
Applying the chain rule shows  $\Div \bv_h\not\in \breve{Q}_h$
for general $\bv_h\in \breve{\bV}_h$ (unless $F_K$ is affine $\forall K\in \mct$),
and simple calculations show the exact enforcement of the divergence--free constraint
and the pressure--robustness of the scheme using $\breve{\bV}_h \times \breve{Q}_h$
is lost on curved elements.  

The methodology we use to construct
divergence-free and pressure robust 
schemes consists of two main ideas.
First, instead of composition, 
we use a divergence--preserving transformation
to recover the divergence--free property, i.e.,
we use a Piola transform 
in the definition of the local velocity space
instead of composition.  Combining the local spaces defined through this mapping
with the Lagrange degrees of freedom yields
a global non-conforming (velocity) finite element space that is $\bH^1$-conforming
in the interior of the domain and $\bH({\rm div})$-conforming globally.
We also show that the resulting space is ``weakly continuous,'' and therefore
suitable for second-order elliptic problems.

The second main idea in our construction is to treat
the Scott-Vogelius pair as a macro-element, rather than
a finite element space defined on a refined (Clough-Tocher) triangulation.
In particular, local spaces are defined by mapping
a macro reference local space, and therefore the corresponding
finite element code does not ``see'' the global Clough-Tocher triangulation.
This modification is motivated by the stability analysis of the Scott-Vogelius pair,
which is based on Stenberg's macro-element technique \cite{Boffi08}.
Adopting this technique to the isoparametric setting, we show that the resulting
pair satisfies the inf-sup condition, and therefore the finite element method
for the Stokes problem is well-posed.

The rest of the paper is organized as follows.  In the next section, we set the notation,
 state the properties of the quadratic diffeomorphisms, and provide some preliminary results.
 In Section \ref{sec:LocalSpaces}, we define the local spaces of the velocity-pressure pair
 and provide a unisolvent set of degrees of freedom.   Here, we also prove a local inf-sup
 stability result.  Section \ref{sec:GlobalSpaces} states the global spaces and proves
 a global inf-sup stability result.  We also show in this section that functions in the discrete velocity space
enjoy weak continuity properties.  In Section \ref{sec:FEM}, we state the finite element method
and show that the method is optimally convergent.  Section \ref{sec:Robust}
gives a pressure-robust scheme through the use of commuting projections,
and Section \ref{sec:Numerics} provides numerical experiments which confirm
the theoretical results.  Some auxiliary results are given in Appendix \ref{sec:ProofPrelim}.

\section{Preliminaries}\label{sec:Prelim}

We assume that the domain $\Omega\subset \bbR^2$ is sufficiently smooth,
and the boundary $\p\Omega$ is given by a finite number of local charts.
The construction of the mesh with curved boundaries follows  the standard isoparametric framework in \cite{Lenoir86,Brenner,CiarletRaviart72,Bernardi86}.
In particular, we start with a shape-regular and affine triangulation $\tilde \calT_h$, with mesh size sufficiently small, such that the boundary vertices of $\tilde \calT_h$
lie on $\p\Omega$, and 
$\tilde \Omega_h:={\rm int}\Big(\cup_{\tilde T\in \tilde \calT_h} \overline{\tilde T}\Big)$ is an $\mathcal{O}(h^2)$
polygonal approximation to $\Omega$.   Here, $h = \max_{\tilde T\in \tilde \calT_h} {\rm diam}(\tilde T)$.
We assume each $\tilde T\in \tilde \calT_h$ has at most two boundary vertices.

\begin{remark}
For the continuation of the paper, we use $C$ (with or without subscript) to denote a 
generic constant that is independent of any mesh size parameter.
\end{remark}

We let 
$G:\tilde\Omega_h \to \Omega$ be a bijective map with $\|G\|_{W^{1,\infty}(\tilde \Omega_h)}\le C$ such that
such that $G|_{\tilde T}(x) = x$ at all vertices of $\tilde T$, in particular,
$G$
 is the identity map for any triangle $\tilde T\in \tilde \mct$
with three interior vertices. 
We denote by $G_h$ the piecewise quadratic nodal interpolant of $G$ satisfying 
$\|DG_h\|_{W^{1,\infty}(\tilde T)}\le C$ and $\|DG_h^{-1}\|_{W^{1,\infty}(\tilde T)}\le C$ for all $\tilde T\in \tilde \calT_h$.
We then 
set 
\[
\mct = \{G_h(\tilde T):\ \tilde T\in \tilde \calT_h\},\qquad  \Omega_h:={\rm int}\Big(\cup_{ T\in  \calT_h} \overline{ T}\Big)
\]
to be the isoparametric triangulation and computational domain, respectively.

Denote by $\hat T$ the reference triangle with vertices $(1,0),(0,1)$, and $(0,0)$.
For $\tilde T\in \tilde \calT_h$, we denote
by $F_{\tilde T}:\hat T\to \tilde T$ an affine mapping satisfying $|F_{\tilde T}|_{W^{1,\infty}(\hat T)}\le C h_T$
and $|F_{\tilde T}^{-1}|_{W^{1,\infty}(\tilde T)}\le C h_T^{-1}$, where $h_T = {\rm diam}(\tilde T)$.
We define the quadratic diffeomorphism $F_T:\hat T\to T$
as $F_T = G_h\circ F_{\tilde T}$ which satisfies
\begin{align}\label{eqn:FProperties}
&|F_T|_{W^{m,\infty}(\hat T)}\le C h_T^m\quad 0\le m\le 2,\qquad |F_T^{-1}|_{W^{m,\infty}(T)}\le C h_T^{-m}\quad 0\le m\le 3,\\
&\nonumber c_1 h_T^2 \le \det(DF_T)\le c_2 h_T^2,
\end{align}
where $h_T = {\rm diam}(G_h^{-1}(T))$.
Note the mappings
$F_T$ and $F_{\tilde T}$ (with $T = G_h(\tilde T)$) are oriented in the same way so that $F_T = F_{\tilde T}$ at the vertices 
of $\hat T$.  In particular, the mappings coincide 
 if $G|_{\tilde T}$ is the identity operator.
 Furthermore, if $e\subset \p T$ is a straight edge with $e = F_{T}(\hat e)$ and $\hat e\subset \p \hat T$,
then $F_T|_{\hat e}$ is affine.  If $T\in \mct$ has all straight edges, then $F_T$ is affine and $T = G_h(\tilde T) = \tilde T$.
The conditions on $F_T$ and the shape-regularity of $\tilde \calT_h$ imply 
$|T|/|G_h^{-1}(T)|\le C$ and $|G_h^{-1}(T)|/|T|\le C$ for all $T\in \calT_h$.

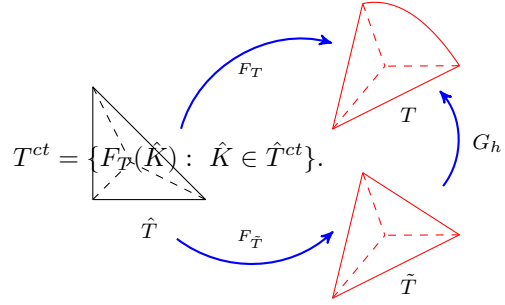
\begin{wrapfigure}{R}{0.48\textwidth}
\centering
\begin{tikzpicture}[scale=0.75]
\draw[-](0,0)--(2,0);
\draw[-](2,0)--(0,2);
\draw[-](0,2)--(0,0);
\draw[-,dashed](0,0)--(2/3,2/3);
\draw[-,dashed](2,0)--(2/3,2/3);
\draw[-,dashed](0,2)--(2/3,2/3);
    \draw (1,-0.5) node {\footnotesize $\hat{T}$};
\draw[red, domain=4.785:6.5,smooth,variable=\t]plot (\t,-.5*\t^2+5.0*\t-9);
\draw[-,red](4.785,3.4768875)--(4.25,1.25);
\draw[-,red](4.25,1.25)--(6.5,2.375);
\draw[-,dashed,red](4.785,3.4768875)--(5.178333333,2.367295833);
\draw[-,dashed,red](4.25,1.25)--(5.178333333,2.367295833);
\draw[-,dashed,red](6.5,2.375)--(5.178333333,2.367295833);
\draw[-,red](4.785,0.4768875)--(4.25,-1.75);
\draw[-,red](4.25,-1.75)--(6.5,-0.625);
\draw[-,red](4.785,0.4768875)--(6.5,-0.625);
\draw[-,dashed,red](4.785,0.4768875)--(5.178333333,2.367295833-3);
\draw[-,dashed,red](4.25,1.25-3)--(5.178333333,2.367295833-3);
\draw[-,dashed,red](6.5,2.375-3)--(5.178333333,2.367295833-3);

 \node (phys) at (4.5,2.75) {};
 \node (ref)  at (1.5,1) {}
 edge[pil,bend left=45,blue] (phys.west);
 \draw (2.8,2.3) node {\tiny$F_{T}$};
  \draw (5.6,1.5) node {\footnotesize ${T}$};
  \draw (2.8,2.3-3.) node {\tiny$F_{\tilde T}$}; 
  \draw (5.6,1.5-3) node {\footnotesize ${\tilde T}$};  
  
  \node (physb) at (4.5,2.75-3.25) {};
 \node (refb)  at (1.25,-0.5) {}
 edge[pil,bend right=40,blue] (physb.west);
 
   \node (physc) at (6.25,2.) {};
 \node (refc)  at (6.,0) {}
 edge[pil,bend right=40,blue] (physc.west);
 
   \draw (7,1) node {\footnotesize $G_h$};  
\end{tikzpicture}
\caption{\label{fig:mappings}\noindent Left: Clough-Tocher split of the reference triangle $\hat T$.  Right: The
corresponding curved and straight macro elements induced by the mappings $F_T$ and $F_{\tilde T}$.}
\end{wrapfigure}

Denote by $\hat T^{ct} = \{\hat K_i\}_{i=1}^3$ the Clough--Tocher triangulation of
the reference triangle, obtained by connecting the vertices
of $\hat T$ with its barycenter.  We then define the analogous local triangulations
on $\tilde T\in \tilde \calT_h$ and $T\in \calT_h$, respectively, (cf. Figure \ref{fig:mappings})
\begin{align*}
\tilde T^{ct}& = \{F_{\tilde T}(\hat K): \hat K\in \hat T^{ct}\},\qquad
T^{ct}  = \{F_T(\hat K):\ \hat K\in \hat T^{ct}\}.
\end{align*}
{The properties of $F_T$ 
show  $|T|\le C|K|$ for all $K\in T^{ct}$.}

We denote by $\calE^I_h$ the interior (straight) edges of $\calT_h$, and by $\calE_h^{I,\p}\subset \calE_h^I$ the 
set of interior edges that have one endpoint on $\p\Omega_h$, i.e.,
the set of interior edges that ``touch'' the computational boundary.
We use the generic $\bn$ to denote a outward unit normal of a domain which is clear from its
context.  The tangent vector $\bt$ is obtained by rotating $\bn$ $90$ degrees counterclockwise.

{
\begin{remark}
\begin{enumerate}

\item 
The globally refined triangulations are given by
\begin{align*}
\tilde \calT_h^{ct}& = \{\tilde K:\ \tilde K\in \tilde T^{ct},\ \exists \tilde T\in \tilde \calT_h\},\\
 \calT_h^{ct}& = \{ K:\  K\in  T^{ct},\ \exists  T\in  \calT_h\}.
 \end{align*}
 However, we emphasize that the construction of the Clough-Tocher 
isoparametric mesh $\calT_h^{ct}$ is constructed by mapping
the reference {\em macro element} $\hat T^{ct}$.
%
In particular, the finite element spaces, given in  subsequent sections, are defined
on $\mct$ (not $\mct^{ct}$); in fact, the corresponding finite element code does not ``see''
the refined triangulation $\mct^{ct}$.

\item Note that this construction leads
to curved interior edges in $\calT_h^{ct}$, as interior edges of $T^{ct}$ may be curved.
\end{enumerate}
\end{remark}
}

The proofs of the following two lemmas are given in Appendix \ref{sec:ProofPrelim}.
\begin{lemma}\label{lem:ATLem}
For each $T\in \mct$, define the matrix valued function $A_T:\hat T\to \mathbb{R}^{2\times 2}$ as
\begin{align}
\label{eqn:ATDef}
A_T(\hat x) = \frac{DF_T(\hat x)}{\det(DF_T(\hat x))}.
\end{align}
Then there holds
\begin{align}
\label{eqn:ATbound}
|A_T|_{W^{m,\infty}(\hat T)}\le C h_T^{m-1},\qquad \text{and}\qquad
|A_T^{-1}|_{W^{m,\infty}(\hat T)}
\le \left\{
\begin{array}{ll}
Ch_T^{1+m} & m=0,1\\
0 & m\ge 2
\end{array}
\right.
\end{align}
\end{lemma}

\begin{lemma}\label{lem:cofactorOnEdge}
Let $\tilde T\in \tilde \calT_h$ and $T\in \calT_h$ with $T = G_h(\tilde T)$.
Let $\hat e$ be an edge of $\hat T$ with outward unit normal $\hat \bn$,
and assume that the corresponding edge $e = F_T(\hat e)$ on $T$ is straight.
Then 
\[
\det(DF_T(\hat x))(DF_T(\hat x))^{-\intercal}\hat \bn = \det(DF_{\tilde T} (\hat x))(DF_{\tilde T}(\hat x))^{-\intercal}\hat \bn
\]
is constant on $\hat e$.
\end{lemma}

We also need a scaling result which is found in \cite{Bernardi86}.
\begin{lemma}\label{lem:scaling}
Suppose that $\bw(x) = \hat \bw(\hat x)$ for sufficiently smooth $\bw\in W^{m,p}(T)$.
Then for any $K\in T^{ct}$,
\begin{align*}
|\bw|_{W^{m,p}(K)}&\le C h_T^{2/p-m} \sum_{r=0}^m h_T^{2(m-r)} |\hat \bw|_{W^{r,p}(\hat K)},\\
|\hat \bw|_{W^{m,p}(\hat K)} &\le C h_T^{m-2/p} \sum_{r=0}^m |\bw|_{W^{r,p}(K)},
\end{align*}
with $\hat K = F_T^{-1}(K)$.
\end{lemma}

\section{Local Spaces}\label{sec:LocalSpaces}
Recall $\hat T\subset \bbR^2$ is the
reference triangle,
and  $\hat T^{ct} = \{\hat K_1,\hat K_2,\hat K_3\}$
is the Clough--Tocher triangulation, obtained 
by connecting the vertices
of $\hat T$ with its barycenter.  We define the polynomial 
spaces on $\hat T$ {without boundary conditions}:
\begin{align*}
\hat \bV & = \{\hat \bv\in \bH^1(\hat T):\ \hat \bv|_{\hat K}\in \bpol_{2}(\hat K)\ \forall \hat K\in \hat T^{\rm ct}\},\quad
\hat Q  = \{\hat q\in L^2(\hat T):\ \hat q|_{\hat K} \in \pol_{1}(\hat K)\ \forall \hat K\in \hat T^{\rm ct}\},
\end{align*}
where $\pol_k(S)$ is the space of scalar polynomials of degree $\le k$ with domain $S$, and $\bpol_k(S) = [\pol_k(S)]^2$.

For an affine triangle $\tilde T\in \tilde \calT_h$ in the polygonal mesh,
we define the  spaces via composition
\begin{align*}
\tilde \bV(\tilde T) &= \{\tilde \bv\in \bH^1(\tilde T):\ \tilde \bv(\tilde x) = \hat \bv(\hat x),\ \exists \hat \bv\in \hat \bV \},\quad
\tilde Q(\tilde T)  = \{\tilde q\in L^2(\tilde T):\ \tilde q(\tilde x) = \hat q(\hat x),\ \exists \hat q\in \hat Q\},
\end{align*}
where $\tilde x = F_{\tilde T}(\hat x)$.  Thus, $\tilde \bV(\tilde T)$ is
the local, quadratic Lagrange finite element space with respect
to $\tilde T^{ct}$, and $\tilde Q(\tilde T)$
is the space of (discontinuous) piecewise linear polynomials with
respect to $\tilde T^{ct}$.  We also define the analogous spaces
with boundary conditions
\begin{alignat*}{2}
&\hat \bV_0 = \hat \bV\cap \bH^1_0(\hat T),\qquad &&\hat Q_0 = \hat Q\cap L^2_0(\hat T),\\
&\tilde  \bV_0(\tilde T) = \tilde \bV(\tilde T)\cap \bH^1_0(\tilde T),\qquad &&\tilde  Q_0(\tilde T)  = \tilde Q(\tilde T)\cap L^2_0(\tilde T).
\end{alignat*}

For  $T\in \mct$, possibly with curved boundary,
we define the spaces with the aid of the Piola transform
\begin{alignat*}{2}
\bV(T)& = \{\bv\in \bH^1(T):\ \bv(x) = A_T(\hat x) \hat \bv(\hat x),\ \exists \hat \bv\in \hat \bV\},\ \ &&\bV_0(T)  = \bV(T)\cap \bH^1_0(T),\\
Q(T) & = \{q\in L^2(T):\ q(x) = {\hat q(\hat x)},\ \exists \hat q\in \hat Q\},\quad &&
Q_0(T)  = \{q\in L^2(T):\ q(x) = {\hat q(\hat x)},\ \exists \hat q\in \hat Q_0\}.
%
\end{alignat*}
Here, $x = F_T(\hat x)$ and we recall $A_T(\hat x) = DF_T(\hat x)/\det(DF_T(\hat x))$.
If $F_T$ is affine, then  $\bV(T) = \tilde \bV(\tilde T)$ and $Q(T) = \tilde Q(\tilde T)$;
otherwise,
both $\bV(T)$ and $Q(T)$ are not necessarily piecewise polynomial spaces.
Moreover, for $\bv\in \bV(T)$ and for a straight edge $e\subset \p T$,
the restriction of $\bv$ to $e$ is not necessarily a polynomial, even though
$F^{-1}_T$ is affine on $e$.  Nonetheless, the next lemma shows
 the normal component of $\bv$ is a polynomial on straight edges.

\begin{lemma}
Let $\bv\in \bV(T)$, and suppose that $e$ is a straight edge of $\p T$ with unit normal $\bn$.
Then $\bv\cdot \bn|_e$ is a quadratic polynomial.
\end{lemma}
\begin{proof}
Write 
$\bv(x) = A_T(\hat x) \hat \bv(\hat x)$ for some $\hat \bv\in \hat \bV$,
and set $\hat e = F_T^{-1}(e)$ to be the corresponding edge in $\p \hat T$ with outward unit normal $\hat \bn$.
We then have
\begin{align*}
\hat \bv\cdot \hat \bn
& = (\det(DF_T) DF_T^{-1} \bv)\cdot \hat \bn
 = (\det(DF_T)DF_T^{-\intercal} \hat \bn) \cdot \bv.
\end{align*}
By Lemma \ref{lem:cofactorOnEdge}, $(\det(DF_T)DF_T^{-\intercal} \hat \bn)$
is a constant vector.  Using the identity
$\bn = {DF^{-\intercal} \hat \bn}/{|DF^{-\intercal} \hat \bn|} $ \cite{MonkBook}, 
we conclude $(\det(DF_T)DF_T^{-\intercal} \hat \bn)$ is a non-zero multiple of $\bn$.
In particular $\bv\cdot \bn$ is a non-zero multiple of $\hat \bv\cdot \hat \bn$.
Because $F_T|_{\hat e}$ is affine and $\hat \bv\cdot \hat \bn$ is a quadratic polynomial on $\hat e$,
we conclude  $\bv\cdot \bn|_e$ is a quadratic polynomial on $e$.
\end{proof}

\begin{lemma}\label{lem:H1Piola}
Suppose $\bv = A_T \hat \bv\in \bV(T)$ for some $\hat \bv\in \hat \bV$.
There holds
$\|\bv\|_{H^1(T)}\le C h_T^{-1} \|\hat \bv\|_{H^1(\hat T)}.$
\end{lemma}
\begin{proof}
By a change of variables, the chain rule,  Lemma \ref{lem:ATLem}, and Lemma \ref{lem:scaling}, we have
\begin{align*}
%
\| \bv\|_{H^1(T)} &\le C (|  A_T \hat \bv|_{H^1(\hat T)}+h_T \|A_T \hat \bv\|_{L^2(\hat T)})\\
&\le C (\|A_T\|_{L^\infty(\hat T)} \|\hat \bv\|_{H^1(\hat T)}
+ \|A_T\|_{W^{1,\infty}(\hat T)} \|\hat \bv\|_{L^2(\hat T)})\le C h_T^{-1}\|\hat \bv\|_{H^1(\hat T)}.
\end{align*}\hfill
\end{proof}

\subsection{Degrees of freedom for $\bV(T)$}
The canonical (nodal) degrees of freedom (DOFs) of the quadratic
Lagrange finite element space on $T^{ct}$ are a given function's values
at the (four) vertices in $T^{ct}$, and its values
at the (six) edge midpoints in $T^{ct}$.  
Here, we
show that these Lagrange DOFs form a unisolvent
set over $\bV(T)$.

Let $\mathcal{N}_{\hat T} := \{\hat a_i\}_{i=1}^{10}$ denote the set of (four) vertices and (six) edge
midpoints in $\hat T^{ct}$.  We let $\mathcal{N}_T:=\{a_i\}_{i=1}^{10}$ and $\mathcal{N}_{\tilde T}:=\{\tilde a_i\}_{i=1}^{10}$ 
be the corresponding
sets on $T^{ct}$ and $\tilde T^{ct}$, respectively, with $a_i = F_T(\hat a_i)$, and $\tilde a_i = F_{\tilde T}(\hat a_i)$.

\begin{lemma}\label{lem:bvnQuad}
A function $\bv\in \bV(T)$ is uniquely determined by the values $\bv(a)$ for all $a\in \mathcal{N}_T$.
\end{lemma}
\begin{proof}
The number of DOFs given is $20$ which matches in the dimension of $\bV(T)$.
Thus, it suffices to show that if $\bv\in \bV(T)$ vanishes on the DOFs, then $\bv\equiv 0$.

Write $\bv(x) = A_T(\hat x) \hat \bv(\hat x)$ for some $\hat \bv\in \hat \bV$.
We then have
\[
 0 = \bv(a) = A_T(\hat a)\hat \bv(\hat a)\qquad \forall a\in \mathcal{N}_T.
 \]
 Because $A_T(\hat a)$ is invertible, we conclude $\hat \bv(\hat a)=0$ for all $\hat a\in \mathcal{N}_{\hat T}$.
 Since $\hat \bv$ is uniquely determined by these values, we conclude
 $\hat \bv\equiv 0$, and therefore $\bv\equiv 0$.
\end{proof}

\begin{lemma}\label{lem:NormDOFs}
There holds, for all $\bv\in \bV(T)$,
\begin{align*}
\|\bv\|_{H^1(T)}^2\le C \sum_{a\in \mathcal{N}_T} |\bv(a)|^2.
\end{align*}
\end{lemma}
\begin{proof}
Again, we write $\bv(x) = A_T(\hat x)\hat \bv(\hat x)$ with $A_T(\hat x) = DF_T(\hat x)/\det(DF_T(\hat x))$
for some $\hat \bv\in \hat \bV$.  By equivalence of norms in a finite dimensional setting, 
and the estimate $\|A_T^{-1}\|_{L^\infty(\hat T)}\le C h_T$, we have
\begin{align*}
\|\hat \bv\|_{H^1(\hat T)}^2
&\le C \sum_{\hat a\in \mathcal{N}_{\hat T}} |\hat \bv(\hat a)|^2
= C \sum_{\hat a \in \mathcal{N}_{\hat T}} |A_T^{-1}(\hat a) A_T(\hat a) \hat \bv(\hat a)|^2\\
&\le Ch_T^2 \sum_{\hat a \in \mathcal{N}_{\hat T}} |A_T(\hat a) \hat \bv(\hat a)|^2
= Ch_T^2 \sum_{ a \in \mathcal{N}_{ T}} |\bv(a)|^2.
\end{align*}
Therefore by Lemma \ref{lem:H1Piola}, 
\begin{align*}
\|\bv\|_{H^1(T)}^2\le C \|A_T \hat \bv\|^2_{H^1(\hat T)}
\le C\|A_T\|_{W^{1,\infty}(\hat T)}^2\|\hat \bv\|_{H^1(\hat T)}^2
\le C h_T^{-2} \|\hat \bv\|_{H^1(\hat T)}^2\le C \sum_{ a \in \mathcal{N}_{ T}} |\bv(a)|^2.
\end{align*}\hfill
\end{proof}

\begin{lemma}\label{lem:Interp}
For $T\in \mct$, let $\bI_T:\bH^3(T)\to \bV(T)$ be uniquely determined by the conditions
\[
(\bI_T \bu)(a) = \bu(a)\qquad \forall a\in \mathcal{N}_T.
\]
Then there holds
\begin{align*}
\|\bu-\bI_T \bu\|_{H^m(T)}\le C h_T^{3-m} \|\bu\|_{H^3(T)}\qquad \forall \bu\in \bH^3(T),\quad m=0,1.
\end{align*}
\end{lemma}
\begin{proof}
Let $\bu\in \bH^3(T)$, and for notational convenience, we set $\bv = \bI_T \bu$.

Write 
\[
\bv(x) = (A_T \hat \bv)(\hat x),\qquad \bu(x) = (A_T \hat \bu)(\hat x)
\]
with $\hat \bv\in \hat \bV$ and $\hat \bu\in \bH^3(\hat T)$.
By definition of $\bI_T \bu$ and the nodal points, we find
\[
(A_T \hat \bv)(\hat a) = (A_T \hat \bu)(\hat a)\qquad \forall \hat a\in \mathcal{N}_{\hat T}.
\]
Therefore, because $A_T$ is invertible, $\hat \bv(\hat a) = \hat \bu(\hat a)$ for all $\hat a\in \mathcal{N}_{\hat T}$, i.e.,
$\hat \bv$ is the quadratic Lagrange nodal interpolant of $\hat \bu$ with respect 
to the local triangulation $\hat T^{ct}$.  It then follows from standard interpolation theory that
\[
\|\hat \bu - \hat \bv\|_{H^m(\hat T)}\le C |\hat \bu|_{H^3(\hat T)}.
\]
Applying Lemmas \ref{lem:scaling} and \ref{lem:ATLem} then yields
\begin{align*}
|\bu - \bv|_{H^m(T)}
&\le 
C h_T^{1-m} \|A_T(\hat \bu-\hat \bv)\|_{H^m(\hat T)}
\le C h_T^{1-m}\|A_T\|_{W^{m,\infty}(\hat T)} \|\hat \bu-\hat \bv\|_{H^m(\hat T)}
\le C h_T^{-m} |\hat \bu|_{H^3(\hat T)}.
\end{align*}
Finally, we once again use Lemmas \ref{lem:ATLem} and \ref{lem:scaling} to obtain
\begin{align*}
|\hat \bu|_{H^3(\hat T)} 
 = |A_T^{-1} A_T \hat \bu|_{H^3(\hat T)}
&\le C\big(\|A_T^{-1}\|_{L^\infty(\hat T)} |A_T \hat \bu|_{H^3(\hat T)} + |A_T^{-1}|_{W^{1,\infty}(\hat T)} |A_T \hat \bu|_{H^2(\hat T)}\big)\\
&\le C\big(h_T|A_T \hat \bu|_{H^3(\hat T)} +h_T^2 |A_T \hat \bu|_{H^2(\hat T)}\big)
\le C h_T^3 \|\bu\|_{H^3(T)}.
\end{align*}\hfill
\end{proof}

\subsection{A connection between local finite element spaces}
In this section, we explicitly identify
a correspondence between piecewise polynomials
defined on the affine local triangulation $\tilde T^{ct}$
and functions on $T^{ct}$ with $T = G_h(\tilde T)$.
This connection will be used to prove global
inf-sup stability in the subsequent section.

\begin{definition}\label{def:OpDef}
Let $\tilde T\in \tilde \calT_h$ and $T\in \calT_h$ with $T = G_h(\tilde T)$.
\begin{enumerate}
\item We define the operator $\Psi_T:\tilde \bV(\tilde T)\to \bV(T)$ uniquely 
by the conditions
\begin{alignat*}{2}
(\Psi_T \tilde \bv)(a) = \tilde \bv(\tilde a)\qquad \forall \tilde a\in \mathcal{N}_{\tilde T},\quad \text{where } a = G_h(\tilde a).
\end{alignat*}
\item We define the operator $\Upsilon_T:\tilde Q(\tilde T)\to Q(T)$ as
\begin{align*}
(\Upsilon_T \tilde q)(x) =  \tilde q(F_{\tilde T}(\hat x)).
\end{align*}
\end{enumerate}
\end{definition}

\begin{theorem}\label{thm:PsiDef}\
\begin{enumerate}
\item If $F_T$ is affine, then $(\Psi_T\tilde \bv)(x) = \tilde \bv(\tilde x)$, in particular, $\Psi_T$ is the identity operator.
\item {If $e\subset \p T$ is a straight edge, so that $e\subset \p \tilde T$, then
\[
(\Psi_T \tilde \bv)\cdot \bn|_e = \tilde \bv \cdot \bn|_e.
\]
}
\item There holds $\|\Psi_T \tilde \bv\|_{H^1(T)}\le C \|\tilde \bv\|_{H^1(\tilde T)}$. 
\end{enumerate}
\end{theorem}
\begin{proof}
For notational simplicity, we set $\bv = \Psi_T \tilde \bv\in \bV(T)$.

\begin{enumerate}
\item If $F_T$ is affine, so that $DF_T$ is constant, we have $\bV(T) = \tilde \bV(\tilde T)$.
We then conclude that $(\Psi_T\tilde \bv) = \tilde \bv$ by Lemma \ref{lem:bvnQuad}.

\item Let $e\subset \p T$ be a straight edge with outward unit normal $\bn$, endpoints $a_1$ and $a_2$,
and midpoint $a_3$.   
Then $e\subset \p \tilde T$ and
\begin{align*}
(\bv\cdot \bn)(a_1) = (\tilde \bv\cdot \bn)(a_1),\quad
(\bv\cdot \bn)(a_2) = (\tilde \bv\cdot \bn)(a_2),\quad (\bv \cdot \bn)(a_3) = (\tilde \bv \cdot \bn)(a_3).
\end{align*}
By Lemma \ref{lem:bvnQuad}, $\bv\cdot \bn|_e$ and $\tilde \bv\cdot \bn|_e$
are both quadratic polynomials, and therefore, these conditions imply $\bv\cdot \bn|_e = \tilde \bv\cdot \bn|_e$.

\item Set $\hat{\tilde \bv}(\hat x) = \tilde \bv(\tilde x)$ with $\tilde x = F_{\tilde T}(\hat x)$.
Using Lemma \ref{lem:NormDOFs} and a standard scaling argument, we have
\begin{align*}
\|\bv\|_{H^1(T)}^2
&\le C\sum_{a\in \mathcal{N}_T} |\bv(a)|^2 = C\sum_{\tilde a\in \mathcal{N}_{\tilde T}} |\tilde \bv(\tilde a)|^2
 = C\sum_{\hat  a\in \mathcal{N}_{\hat  T}} |\hat{\tilde \bv}(\hat a)|^2\le C \|\hat{\tilde \bv}\|_{H^1(\hat T)}^2\le C \|\tilde \bv\|_{H^1(\tilde T)}^2.
\end{align*}
\end{enumerate}\hfill
\end{proof}

\subsection{Local Inf-sup stability}
In this section, we derive an indirect local inf-sup 
stability result of the pair $\bV_0(T)\times Q_0(T)$.
As a first step, we use the stability of 
the analogous pair $\hat \bV_0\times \hat Q_0$ defined
on the reference triangle.  The proof of the following lemma
is found in, e.g., \cite{ArnoldQin92,GuzmanNeilan18}.
\begin{lemma}\label{lem:RefInfSup}
For any $\hat q \in \hat {Q}_{0}$, 
there exists $\hat \bv \in \hat {\bV}_0$ such that $\hat \nab \cdot \hat \bv=\hat q$ with the bound $\|\hat \bv\|_{H^1(\hat T)}\le C \|\hat q\|_{L^2(\hat T)}$.
\end{lemma}

\begin{theorem}\label{thm:LocalInfSup} Given $q\in {Q}_0(T)$, then
there exists $ \bv\in {\bV}_0(T) $  such that 
\[
{(\nab \cdot \bv)(x) = \frac{h_T^2 q(x)}{\det(DF_T(F_T^{-1}(x)))}},\quad \text{and}\quad\| \bv\|_{H^1(T)}\le C\|q\|_{L^2(T)}.
\]
\end{theorem}
\begin{proof}
Let $q \in {Q}_0(T)$.  Then there exists ${q} \in \hat {Q}_0$ such that 
{$q(x) = {\hat q(\hat x)}$}.  Because $h_T^2 \hat q \in \hat Q_0$,
by Lemma \ref{lem:RefInfSup}, 
there exists {$\hat{\bv} \in \hat{\bV}_0$ such that $\hat{\nab} \cdot \hat{\bv}=h_T^2\hat{q}$} and $\|\hat \bv\|_{H^1(\hat T)}\le C h_T^2 \|\hat q\|_{L^2(\hat T)}$.  Setting
$\bv(x) = A_T\hat \bv \in \bV_0(T)$, we compute
\begin{align*}
(\nab \cdot \bv)(x)=\frac{(\hat{\nab} \cdot \hat{\bv})(\hat{x})}{\det(DF_{T}(\hat x))}={\frac{h_T^2 \hat{q}(\hat{x})}{\det(DF_{T}(\hat x))}= \frac{h_T^2 q(x)}{\det(DF_T(F_T^{-1}(x)))}}.
\end{align*}
Applying Lemmas \ref{lem:RefInfSup} and \ref{lem:H1Piola} and a change of variables
yields 
\begin{align*}
\|\bv\|_{H^1(T)}\le C h_T^{-1} \|\hat \bv\|_{H^1(\hat T)}\le C h_T \|\hat q\|_{L^2(\hat T)} \le C \|q\|_{L^2(T)}.
\end{align*}\hfill
\end{proof}

\section{The Global Spaces}\label{sec:GlobalSpaces}
Define the Scott--Vogelius pair with respect to the affine triangulation $\tilde \calT_h$:
\begin{align*}
\tilde \bV^h & = \{\tilde \bv\in \bH^1_0(\tilde \Omega_h):\ \tilde \bv|_{\tilde T}\in \tilde \bV(\tilde T),\ \forall \tilde T\in \tilde \calT_h\},\qquad
\tilde Q^h  = \{\tilde q\in L^2_0(\tilde \Omega_h):\ \tilde q|_{\tilde T}\in \tilde Q(\tilde T),\ \forall \tilde T\in \tilde \calT_h\}.
\end{align*}
We construct the global spaces $\bV^h\times Q^h$ defined
on $\calT_h$ using the  spaces $\tilde \bV^h\times \tilde Q^h$
and with the aid of the operators $\Psi_T$ and $\Upsilon_T$ given in Definition \ref{def:OpDef}. 
To this end, we define $\Psi$ and $\Upsilon$ to be the operators given by
\[
\Psi|_T = \Psi_T,\qquad \Upsilon|_T = \Upsilon_T\qquad \forall T\in \calT_h.
\]
The global spaces, defined on the isoparametric mesh $\mct$, are then given by
\begin{align*}
\bV^h:&=\{\bv:\ \bv=\Psi \tilde \bv,\ \exists \tilde \bv\in \tilde \bV^h\},\qquad
Q^h: = \{q:\ q = \Upsilon \tilde q,\ \exists \tilde q\in \tilde Q^h\}.
\end{align*}
\begin{remark}
It is easy to see that the space $\bV^h$ is equivalently defined
as functions locally in $\bV(T)$ on each $T\in \mct$, are continuous
on the DOFs in Lemma \ref{lem:bvnQuad}, and vanish on $\p\Omega_h$.
\end{remark}

\begin{theorem}\label{thm:GlobalProperties}\
\begin{enumerate}
\item There holds $\bV^h\subset \bH_0({\rm div};\Omega_h) = \{\bv\in \bL^2(\Omega_h):\ \nab \cdot \bv\in L^2(\Omega_h),\ \bv\cdot \bn|_{\p\Omega_h}=0\}$.

\item There holds $q\in Q^h$ if and only if $q|_T\circ F_T \in \hat Q$ for all $T\in \mct$, and
\begin{align*}
\sum_{T\in \mct} 2|\tilde T| \int_T \frac{q}{\det(DF_T\circ F_T^{-1})} = 0.
\end{align*}
\end{enumerate}
\end{theorem}
\begin{proof}\
\begin{enumerate}
\item Let $T_1,T_2\in \mct$ such that $\emptyset\neq \p T_1\cap \p T_2=:e$,
and let $\bn$ be a unit normal of $e$. Note that $e$ is a straight edge in $\mct$.
Let $\bv = \Psi(\tilde \bv)$ for some $\tilde \bv\in \tilde \bV^h$, and denote by $\bv_i$ the restriction of $\bv$ to $T_i$.
Likewise, let $\tilde \bv_i$ denote the restriction of $\tilde \bv$ to $\tilde T_i$.
Then by Theorem \ref{thm:PsiDef} and the continuity of $\tilde \bv$, we have
\[
\bv_1\cdot \bn|_e = \tilde \bv_1\cdot \bn|_e = \tilde \bv_2\cdot \bn|_e = \bv_2\cdot \bn|_e.
\]
Thus, the normal component of $\bv$ is single-valued on interior edges.
Because $\bv|_{\p T\cap \p \Omega_h}=0$ for all $T\in \mct$,
we  conclude that $\bv\in \bH_0({\rm div};\Omega_h)$.

\item 
Let $q\in Q^h$.  Then there exists a (unique) $\tilde q\in \tilde Q^h$ such that
$q = \Upsilon \tilde q$, with $q|_T(F_T(\hat x)) = \tilde q|_{\tilde T} (F_{\tilde T}(\hat x))$.
We then find by a change of variables
\begin{align*}
0 &= \int_{\tilde \Omega_h} \tilde q = \sum_{\tilde T\in \tilde \mct} \int_{\tilde T} \tilde q
 = \sum_{\tilde T\in \tilde \mct} 2|\tilde T| \int_{\hat T} \tilde q\circ F_{\tilde T}
= \sum_{ T\in  \mct} 2|\tilde T| \int_{\hat T}  q\circ F_{ T} =  \sum_{ T\in  \mct} 2|\tilde T| \int_{ T}  \frac{q}{\det(DF_T \circ F^{-1}_{ T})}.
\end{align*}
The converse is proved similarly.
\end{enumerate}\hfill
\end{proof}

\subsection{Global inf-sup stability}

In this section, we show  the finite element pair $\bV^h\times Q^h$ 
is inf-sup stable.   This is achieved by using the local stability result given in Theorem \ref{thm:LocalInfSup}
combined with Stenberg's macro element technique.

We define the spaces of piecewise constants with respect to $\tilde \mct$ and $\mct$:
\begin{align*}
\tilde Y^h:&=\{q\in L^2_0(\tilde \Omega_h):\ \tilde q|_T\in \pol_0(\tilde T)\ \forall \tilde T\in \tilde \mct\}\subset \tilde Q^h,\qquad
Y^h:=\{q:\ q = \Upsilon(\tilde q),\ \exists \tilde q\in \tilde Y^h\} \subset Q^h.
\end{align*}
We first show that the pair $\bV^h\times Y^h$ is stable in
the following lemma.

\begin{lemma}\label{lem:infsupPrelim}
There holds
\begin{align*}
\sup_{\bv\in \bV^h\backslash \{0\}} \frac{\int_{\Omega_h} (\nab \cdot \bv)q}{\|\nab \bv\|_{L^2(\Omega_h)}} \ge \gamma_1 \| q\|_{L^2(\Omega_h)}\qquad
\forall q\in Y^h,
\end{align*}
where the gradient of $\bv$ is understood piecewise with respect to $\mct$.  Here, $\gamma_1>0$ is a constant independent of $h$.
\end{lemma}
\begin{proof}
 Fix $q\in Y^h$, and let $\tilde q\in \tilde Y^h$ be the piecewise constant function such that $q = \Upsilon \tilde q$.
 Note that, because $q$ and $\tilde q$ are both piecewise constant, there holds $q|_{\Omega_h\cap \tilde \Omega_h} = \tilde q|_{\Omega_h\cap \tilde \Omega_h}$.  In particular, we have
 \[
 \int_T q = \frac{|T|}{|\tilde T|} \int_{\tilde T} \tilde q,\quad \text{and}\quad  \|q\|_{L^2(T)}^2 = \frac{|T|}{|\tilde T|} \|\tilde q\|_{L^2(\tilde T)}^2\qquad \forall \tilde T \in \tilde \calT_h,
 \]
 with $T = G_h(\tilde T)$.	 Thus we have $\|q\|_{L^2(\Omega_h)}\le C\|\tilde q\|_{L^2(\tilde \Omega_h)}$.
 
Let $\tilde \bw\in \bH^1_0(\tilde\Omega_h)$ satisfy $\tilde \nab \cdot \tilde \bw = \tilde q$ and $\|\tilde \nab \tilde \bw\|_{L^2(\tilde\Omega_h)}\le C\|\tilde q\|_{L^2(\tilde \Omega_h)}$.  The results in \cite[Theorem 4.4]{BernardiEtal16} and the properties of $G$ ensure that
$C>0$ is independent of $h$.
From the stability proof of the piecewise 
quadratic-constant pair \cite{BernardiRaugel,Boffi08}, there exists $\tilde \bv\in \tilde \bV^h$ such that
\[
\int_{\tilde e} \tilde \bv = \int_{\tilde e} \tilde \bw,\quad \text{and}\quad \|\tilde \nab \tilde \bv\|_{L^2(\tilde\Omega_h)}\le C \|\tilde \nab \tilde \bw\|_{L^2(\tilde \Omega_h)}.
\]
Let $\bv = \Psi \tilde \bv $ and note that $\|\nab \bv\|_{L^2(\Omega_h)}\le C\|\nab  \tilde \bv\|_{L^2(\tilde \Omega_h)}$
by Theorem \ref{thm:PsiDef} (item (3)).  Furthermore, this theorem shows that, on each $T\in \mct$,
\begin{align*}
\int_T \nab \cdot \bv = \int_{\p T} (\bv\cdot \bn) = \int_{\p \tilde T} (\tilde \bv\cdot \tilde \bn) = \int_{\p \tilde T} (\tilde \bw\cdot \tilde \bn) 
= \int_{\tilde T} \tilde \nab \cdot \tilde \bw = \int_{\tilde T} \tilde q =  \frac{|\tilde T|}{|T|} \int_{T} q,
\end{align*}
and therefore, because $q$ is constant on $T$,
\begin{align*}
\int_T (\nab \cdot \bv) q
 =  \frac{|\tilde T|}{|T|} \int_{T} q^2  = \|\tilde q\|_{L^2(\tilde T)}^2.
\end{align*}
Summing over $T\in \mct$ then gets
\begin{align*}
\int_{\Omega_h} (\nab \cdot \bv) q
 &= \|\tilde q\|_{L^2(\tilde \Omega_h)}^2\ge C\|\tilde q\|_{L^2(\tilde \Omega_h)} \|\tilde \nab \tilde \bw \|_{L^2(\tilde \Omega_h)}
 \ge C\|\tilde q\|_{L^2(\tilde \Omega_h)} \|\tilde \nab \tilde \bv \|_{L^2(\tilde \Omega_h)}\\
 &\ge \|\tilde q\|_{L^2(\tilde \Omega_h)} \|\nab \bv\|_{L^2(\Omega_h)}\ge C\|q\|_{L^2(\Omega_h)} \|\nab \bv\|_{L^2(\Omega_h)}.
\end{align*}
Dividing this expression by $ \|\nab \bv\|_{L^2(\Omega_h)}$ gives us the desired estimate.
\end{proof}

\begin{theorem}\label{thm:GlobalInfsup}
There holds
\begin{align*}
\sup_{\bv\in \bV^h\backslash \{0\}} \frac{\int_{\Omega_h} (\nab \cdot \bv)q}{\|\nab \bv\|_{L^2(\Omega_h)}} \ge C \| q\|_{L^2(\Omega_h)}\qquad
\forall q\in Q^h,
\end{align*}
where the gradient of $\bv$ is understood piecewise with respect to $\mct$.
\end{theorem}
\begin{proof}
Let $q\in Q^h$.  For each $T\in \calT_h$, we define $\bar{q}_T\in \pol_0(T)$ such that
\begin{align*}
\int_T \frac{(q-\bar q_T)}{\det(DF_T)} = 0,
\end{align*}
and set $\bar q$ such that $\bar q|_T = \bar q_T$ for all $T\in \mct$.
Then $(q-\bar q)|_T\in {Q}_{0}(T)$ for all $T\in \calT_h$, and $\bar q\in Y^h$.
Consequently, by Theorem \ref{thm:LocalInfSup}, for each $T\in \calT_h$, there exists $\bv_{1,T}\in {\bV}_0(T)$
such that 
\[
\nab \cdot \bv_{1,T} = \frac{h_T^2 (q-\bar q)}{\det(DF_T)},\qquad \|\nab \bv_{1,T}\|\le C \|q-\bar q\|_{L^2(T)}.
\]
Set $\bv_1$ such that $\bv_1|_T= \bv_{1,T}$ for all $T\in \calT_h$.
Then $\bv_1\in \bV^h$ because $\bv_{1,T}|_{\p T} = 0$. 
We also have $\|\nab \bv_1\|_{L^2(\Omega_h)}\le C\|q-\bar q\|_{L^2(\Omega_h)}$, and
\begin{align*}
\int_{\Omega_h} (\nab \cdot \bv_1) (q-\bar q)
& = \sum_{T\in \mct} \int_T (\nab \cdot \bv_1) (q-\bar q)\\
& = \sum_{T\in \mct} \int_T  \frac{h_T^2 |q-\bar q|^2}{\det(DF_T)} \ge c \sum_{T\in \mct} \int_T |q-\bar q|^2\\
& =c  \|q-\bar q\|_{L^2(\Omega_h)}^2\\
&\ge c \|q-\bar q\|_{L^2(\Omega_h)}\|\nab \bv_1\|_{L^2(\Omega_h)}.
\end{align*}
Next, recall  $\bv_1|_{\p T} = \bv_{1,T}|_{\p T} = 0$, and $\bar q$ is constant on each $T$.
Therefore by the divergence theorem,
\begin{align*}
\int_{\Omega_h} (\nab \cdot \bv_1) \bar q = \sum_{T\in \calT_h} \int_T (\nab \cdot \bv_1) \bar q = \sum_{T\in \calT_h} \int_{\p T} (\bv_1\cdot \bn) \bar q = 0.
\end{align*}
Thus, we conclude the existence of a constant $\gamma_0$ independent of $h$ such that
\begin{align*}
\gamma_0 \|q-\bar q\|_{L^2(\Omega_h)} \le \sup_{\bv\in \bV^h\backslash \{0\}}\frac{\int_{\Omega_h} (\nab \cdot \bv) q}{\|\nab \bv\|_{L^2(\Omega_h)}}.
\end{align*}

Next, we use the stability of the $\bV^h\times Y^h$ pair given in Lemma \ref{lem:infsupPrelim}:
\begin{align*}
\gamma_1 \|\bar q\|_{L^2(\Omega_h)}
& \le \sup_{\bv\in \bV^h\backslash \{0\}}\frac{\int_{\Omega_h} (\nab \cdot \bv) \bar q}{\|\nab \bv\|_{L^2(\Omega_h)}}\\
& \le \sup_{\bv\in \bV^h}\frac{\int_{\Omega_h} (\nab \cdot \bv)  q}{\|\nab \bv\|_{L^2(\Omega_h)}} +\|q-\bar q\|_{L^2(\Omega_h)}
 \le (1+\gamma_0^{-1})\sup_{\bv\in \bV^h\backslash \{0\}}\frac{\int_{\Omega_h} (\nab \cdot \bv)  q}{\|\nab \bv\|_{L^2(\Omega_h)}}.
\end{align*}
Therefore,
\begin{align*}
\|q\|_{L^2(\Omega_h)}
&\le \|q-\bar q\|_{L^2(\Omega_h)}+\|\bar q\|_{L^2(\Omega_h)}
\le (\gamma_0^{-1}+\gamma_1^{-1}(1+\gamma_0^{-1})) \sup_{\bv\in \bV^h\backslash \{0\}}\frac{\int_{\Omega_h} (\nab \cdot \bv)  q}{\|\nab \bv\|_{L^2(\Omega_h)}}.
\end{align*}
This is the desired inf-sup condition.
\end{proof}

\subsection{$\bV^h$ as an approximate $\bH^1_0(\Omega_h)$ function space}
Recall from Theorem \ref{thm:GlobalProperties} 
that the discrete velocity space is $\bH_0({\rm div};\Omega_h)$-conforming.
However, in general there holds $\bV^h\not\subset \bH^1(\Omega_h)$
because $\bv|_e$ is not a quadratic polynomial on all edges $e$ in $\mct$.
More precisely, $\bv|_e$ is not a quadratic polynomial if $e\subset \p T$,
and $T$ has a curved edge (otherwise $\bv|_e$ is quadratic and is continuous across the edge). 
Nonetheless, the definition of $\bV^h$ shows that functions in $\bV^h$ are single-valued
at three points on each internal edge.  We use this property in the next two lemmas
to show that functions in $\bV^h$ are ``weakly continuous.''

\begin{lemma}\label{lem:EhApprox}
There exists an operator $\bE_h:\bV^h\to \bH^1_0(\Omega_h)$ such that for all $\bv\in \bV^h$,
\begin{align}\label{eqn:EhApprox}
\|\bv-\bE_h \bv\|_{L^2(T)}+h_T \|\nab (\bv-\bE_h \bv)\|_{L^2(T)}\le C h_T^2 \|\nab \bv\|_{L^2(T)}\qquad \forall T\in \mct.
\end{align}
\end{lemma}
\begin{proof}
For given $\bv\in \bV^h$, 
there exists $\tilde \bv\in \tilde \bV^h$ such that
$\bv = \Psi \tilde \bv$.  In particular, $\tilde \bv$ is uniquely determined by
\begin{alignat*}{2}
&\bv|_T(a) = \tilde \bv|_{\tilde T}(\tilde a)\qquad \forall a\in \mathcal{N}_T,\quad \forall T\in \mct,
\end{alignat*}
with $T = G_h(\tilde T)$.
We define the function $\bE_h \bv$ via
\begin{align*}
\bE_h\bv|_T=
 (\tilde \bv \circ  F_{\tilde T} \circ F_{T}^{-1})|_T\qquad \forall T\in \mct.
\end{align*}
That is, $\bE_h \bv$ is the function in the standard isoparametric quadratic Lagrange finite element
space associated with $\tilde \bv$.  We then have $\bE_h \bv\in \bH^1_0(\Omega_h)$. 
Furthermore, since $F_T^{-1}$ is affine on straight edges, we see 
$\tilde \bv = \bE_h \bv$ on straight edges.  In particular, we conclude 
\begin{alignat*}{2}
\bE_h\bv |_T(a) = \bv|_T(a)\qquad \forall a\in \mathcal{N}_T,\quad \forall T\in \mct. 
\end{alignat*}

We now estimate the difference $\bv- \bE_h \bv$.
On affine (non-curved) triangles, we easily see
 $\bv = \bE_h \bv$ because both functions are piecewise quadratic 
polynomials.  Therefore the estimate trivially holds in this case.

Next, let $T\in \mct$ with curved boundary.
We then have $\bv|_{\p T\cap \p\Omega_h}=0$.
Write
$\bv|_T(x) = A_T(\hat x) \hat \bv(\hat x)$ with $\hat \bv\in \hat \bV$
and $A_T = DF_T/\det(DF_T)$.  We also set $\hat \bw\in \hat \bV$
such that $\hat \bw(\hat x) = \bE_h\bv|_T (x)$. 
We then have
\begin{align*}
A_T(\hat a) \hat \bv(\hat a) = \hat \bw(\hat a)\qquad \forall \hat a\in \mathcal{N}_{\hat T}.
\end{align*}
Thus, $\hat \bw$ is the piecewise quadratic Lagrange interpolant of $A_T \hat \bv$ on $\hat T^{ct}$.
It then follows from the Bramble--Hilbert lemma that
\[
\|A_T\hat \bv - \hat \bw\|_{H^m(\hat K)}\le C |A_T\hat \bv|_{H^3(\hat K)}\qquad \forall \hat K\in \hat T^{ct},\quad m=0,1.
\]
Expanding the right--hand side,  using Lemma \ref{lem:ATLem}, and the fact that $\hat \bv|_{\hat K}$ is a quadratic polynomial shows
\begin{align*}
 |A_T\hat \bv|_{H^3(\hat K)}
 &\le C\big(|A|_{W^{3,\infty}(\hat K)} \|\hat \bv\|_{L^2(\hat K)}+|A|_{W^{2,\infty}(\hat K)} |\hat \bv|_{H^1(\hat K)}+|A|_{W^{1,\infty}(\hat K)} |\hat \bv|_{H^2(\hat K)}\big)\\
 &\le C \|\hat \bv\|_{H^2(\hat K)}\le C \|\hat \bv\|_{L^2(\hat K)},
 \end{align*}
 where we used the equivalence of norms in a finite dimensional setting in the last inequality.
 Using the estimate $\|A_T^{-1}\|_{L^\infty(\hat T)}\le Ch_T$, we conclude
 \begin{align*}
\|A_T \hat \bv-\hat \bw\|_{H^m(\hat T)}\le C h_T \|A_T \hat \bv\|_{L^2(\hat T)},\qquad m=0,1.
 \end{align*}
 We then use Lemma \ref{lem:scaling} and the Poincare inequality to get ($m=0,1$)
 \begin{align*}
 \|\bv-\bE_h \bv\|_{H^m(T)}
 &\le C h_T^{1-m}\|A_T \hat \bv-\hat \bw\|_{H^m(\hat T)}\\
 &\le C h_T^{2-m} \|A_T \hat \bv\|_{L^2(\hat T)} \\
& \le C h_T^{1-m} \|\bv\|_{L^2(T)}\le C h_T^{2-m}\|\nab \bv\|_{L^2(T)}.
 \end{align*}\hfill
\end{proof}

Recall, $\mathcal{E}_h^I$ is the set of internal edges of $\mct$.  
For $e = \mathcal{E}_h^I$, write $e = \p T_+\cap \p T_-$ for some $T_{\pm}\in \mct$.
Let $\bn_{\pm}$ denote the outward unit normal of $\p T_{\pm}$ restricted to $e$,
and for a piecewise smooth function $\bv$, let $\bv_{\pm}$ denote the restriction of $\bv$ to $T_{\pm}$.  
We then define
the jump operator
\[
[\bv]|_e = \bv_+ \otimes \bn_+ + \bv_- \otimes \bn_-,
\]
where $({\bm a}\otimes {\bm b})_{i,j} = a_i b_j$.

\begin{lemma}\label{lem:JumpEst}
Let $e\in \mathcal{E}_h^I$ with $e = \p T_+\cap \p T_-$ for some $T_{\pm}\in \mct$.
Then there holds for all $\bv\in \bV^h$,
\begin{align}\label{eqn:JumpEst}
\Big|\int_e [\bv]\Big|\le C h_T^3(\|\nab \bv\|_{L^2(T_+)}+\|\nab \bv\|_{L^2(T_-)}),
\end{align}
where $h_T = \max\{h_{T_+},h_{T_-}\}$.
\end{lemma}
\begin{proof}
Let $a_1,a_2$ be the endpoints of $e$, and let $a_3$ be the midpoint of $e$.
If both $T_+$ and $T_-$ are affine,
then $\bv|_{T_+\cup T_-} \in \bH^1(T_+\cup T_-)$. This implies $[\bv]|_e=0$,
and the estimate trivially follows.

Next suppose that at least one of $T_{\pm}$ has a curved edge,
which implies that one of the endpoints $a_1,a_2$ lie on $\p\Omega_h$.
Without loss of generality,
we assume that $T_+$ has a curved edge.

By construction of the space $\bV^h$, in particular, the definition of $\Psi$, we have $[\bv]|_e(a_i)=0$, $i=1,2,3$.
It then follows from the error of Simpson's rule that
\begin{align}\label{eqn:Simpson}
\Big|\int_e [\bv]\Big|\le C |e|^5 \big|[\bv]\big|_{W^{4,\infty}(e)}\le C h_T^5(|\bv|_{W^{4,\infty}(K_+)} + |\bv|_{W^{4,\infty}(K_-)}),
\end{align}
where $K_{\pm}\in T^{ct}_{\pm}$ satisfy $\p K_+\cap \p K_- = e$.

Write $\bv|_{K_{\pm}}(x) = (A_{T_{\pm}} \hat \bv_{\pm})|_{\hat K_{\pm}}(\hat x)$ with $\hat \bv_{\pm}\in \hat \bV$,
and $\hat K_{\pm} = F_{T_{\pm}}^{-1}(K_{\pm})$.
We  apply Lemmas \ref{lem:scaling} and \ref{lem:ATLem} to the right-hand side of \eqref{eqn:Simpson}
and use the fact that $\bv_{\pm}$ is a quadratic polynomial:
\begin{align*}
|\bv|_{W^{4,\infty}(K_{\pm})}
&\le Ch_{T_{\pm}}^{-4} \sum_{r=0}^4 h_{T_{\pm}}^{2(4-r)} |A_{T_{\pm}} \hat \bv_{\pm}|_{W^{r,\infty}(\hat K_{\pm})}\\
&\le C h_{T_{\pm}}^{4} \sum_{r=0}^4 h_{T_{\pm}}^{-2r} \sum_{j=0}^r  |A_{T_{\pm}}|_{W^{r-j,\infty}(\hat T_{\pm})} |\hat \bv_{\pm}|_{W^{j,\infty}(\hat K_{\pm})}\\
&\le C h_{T_{\pm}}^{4} \sum_{r=0}^4 h_{T_{\pm}}^{-2r} \sum_{j=0}^2  h_{T_{\pm}}^{r-j-1} |\hat \bv_{\pm}|_{W^{j,\infty}(\hat K_{\pm})}\\
&\le C  \sum_{j=0}^2  h_{T_{\pm}}^{-j-1} |\hat \bv_{\pm}|_{W^{j,\infty}(\hat K_{\pm})}\le C h_{T_{\pm}}^{-3} \|\hat \bv_{\pm}\|_{L^2(\hat K_{\pm})},
\end{align*}
where we used equivalence of norms in the last inequality.
Using the estimate $\|A_{T_{\pm}}^{-1}\|_{L^\infty(\hat T)}\le Ch_{T_{\pm}}$ and Lemma \ref{lem:scaling} we get
\begin{align*}
|\bv|_{W^{4,\infty}(K_{\pm})}\le C h_{T_{\pm}}^{-2} \|A_{T_{\pm}} \hat \bv_{\pm}\|_{L^2(\hat T)}\le C h_{T_{\pm}}^{-3} \|\bv\|_{L^2(T_{\pm})}.
\end{align*}
Combining this estimate with \eqref{eqn:Simpson} and applying the Poincare inequality (on $T_+$) yields
\begin{align}\label{eqn:Simpson2}
\Big|\int_e [\bv]\Big|\le C  h_T^2\big( \|\bv\|_{L^2(T_-)} + h_T \|\nab \bv\|_{L^2(T_+)}\big).
\end{align}

Next we show $\|\bv\|_{L^2(T_-)}\le C h_T \|\nab \bv\|_{L^2(T_-)}$.
To this end, we
set $\bw = \bE_h\bv$, where $\bE_h \bv$ is given in Lemma \ref{lem:EhApprox}.
We then write
\begin{align*}
 \|\bv\|_{L^2(T_-)} \le \|\bv-\bw\|_{L^2(T_-)} + \|\bw\|_{L^2(T_-)} \le \|\bw\|_{L^2(T_-)} + C h_{T_{-}}^2 \|\nab \bv\|_{L^2(T_-)}.
 \end{align*}
 
 Let $\hat \bw\in \hat \bV$ such that $\hat \bw(\hat x) = \bw(x)$ with $x = F_{T_-}(\hat x)$.
 Noting that $\bw$ vanishes on $\p\Omega_h$, in particular $\bw$ vanishes
 on at least one vertex of $T_-$, we conclude that
 \[
 \hat \bw \to \|\hat \nab \bw\|_{L^2(\hat T)}
 \]
 is a norm.  Therefore by Lemma \ref{lem:scaling} and equivalence of norms,
 \begin{align*}
 \|\bw\|_{L^2(T_-)} \le C h_T \|\hat \bw\|_{L^2(\hat T)} \le C h_T \|\hat \nab \hat \bw\|_{L^2(\hat T)} \le C h_T \|\nab \bw\|_{L^2(T_-)}.
 \end{align*}
 Hence, we have
 \begin{align*}
  \|\bv\|_{L^2(T_-)} 
  %
 \le C \Big( h_T \|\nab \bw\|_{L^2(T_-)} +  h_T^2 \|\nab \bv\|_{L^2(T_-)}\Big)
  %
 %
 \le C h_T \|\nab \bv\|_{L^2(T_-)}.
  \end{align*}
  Combining this estimate with \eqref{eqn:Simpson2}, we obtain
  the desired estimate \eqref{eqn:JumpEst}.  This concludes the proof.
\end{proof}

\section{Finite Element Method and Convergence Analysis}\label{sec:FEM}
For a given function ${\bm f}$, we let $(\bu,p)\in \bH^1_0(\Omega)\times L^2_0(\Omega)$ be the solution
to the Stokes problem
\begin{alignat*}{2}
-\nu \Delta \bu + \nab p & = {\bm f},\qquad \nab \cdot \bu = 0\qquad \text{in }\Omega,
\end{alignat*}
where $\nu>0$ is the viscosity.
We assume that $\p\Omega$ and ${\bm f}$ are sufficiently smooth
such that $(\bu,p)\in \bH^3(\Omega)\times H^2(\Omega)$,
and can be extended to $\bbR^2$ in a way such that
$(\bu,p)\in \bH^3(\bbR^2)\times H^2(\bbR^2)$ with $\nab \cdot \bu=0$ and (cf.~\cite{KatoEtAl2000})
\[
\|\bu\|_{H^3(\bbR^2)}\le C \|\bu\|_{H^3(\Omega)},\qquad \|p\|_{H^2(\bbR^2)}\le C \|p\|_{H^2(\Omega)}.
\]
We then extend ${\bm f}$ by
\[
{\bm f} = -\nu \Delta \bu+ \nab p,
\]
so that ${\bm f}\in \bH^1(\bbR^2)$.  

We denote by ${\bm f}_h\in \bL^2(\Omega_h)$ a computable approximation of ${\bm f}|_{\Omega}$.
For example, ${\bm f}_h$ could be the (global) quadratic Lagrange nodal interpolant of ${\bm f}$.

The finite element method seeks $(\bu_h,p_h)\in \bV^h\times Q^h$
such that
\begin{subequations}
\label{eqn:FEM}
\begin{alignat}{2}
\label{eqn:FEM1}
\int_{\Omega_h} \nu \nab \bu_h:\nab \bv - \int_{\Omega_h} (\nab \cdot \bv) p_h & = \int_{\Omega_h} {\bm f}_h\cdot \bv \qquad &&\forall \bv\in\bV^h,\\
\label{eqn:FEM2}
\int_{\Omega_h} (\nab \cdot \bu_h)q & =0\qquad &&\forall q\in Q^h,
\end{alignat}
\end{subequations}
where the gradient is understood piecewise with respect to the triangulation.

By the inf-sup condition established in Lemma \ref{thm:GlobalInfsup} and standard
theory of mixed finite element methods, problem \eqref{eqn:FEM} is well-posed.

\begin{theorem}
There exists a unique solution $(\bu_h,p_h)\in \bV^h\times Q^h$ satisfying
\eqref{eqn:FEM}.
\end{theorem}

Next, we show that, despite the non-inclusion $\nab \cdot \bV^h\not\subset Q^h$,
the finite element method yields exactly divergence--free velocity approximations.
\begin{lemma}\label{lem:DivFree}
Suppose  $\bu_h\in \bV^h$ satisfies \eqref{eqn:FEM2}.
Then $\nab \cdot \bu_h \equiv 0$ in $\Omega_h$.
\end{lemma}
\begin{proof}
For each $T\in \calT_h$, write
%
$\bu_h|_T = A_T \hat \bu_T,$
with $\hat \bu_T\in \hat \bV$.
Define $q$ to be the piecewise function 
\[
q|_T(x) = \frac{1}{2 |\tilde T|}(\hat \nab \cdot \hat \bu_T)(\hat x),\qquad x = F_T(\hat x),\qquad T = G_h(\tilde T)
\]
for all $T\in \mct$.  Well-known properties of the Piola transform show
\begin{align*}
q|_T = \frac{\det(DF_T\circ F_T^{-1})}{2|\tilde T|} (\nab \cdot \bu_h|_T)\qquad \forall T\in \mct,
\end{align*}
and therefore
\begin{align*}
\sum_{T\in \mct} 2|\tilde T| \int_T \frac{q}{\det(DF_T\circ F_T^{-1})}
= \sum_{T\in \mct} \int_T \nab \cdot \bu_h = \int_{\p \Omega_h} \bu_h\cdot \bn = 0.
\end{align*}
Thus we conclude $q\in Q^h$ by Theorem \ref{thm:GlobalProperties}.

Next, by \eqref{eqn:FEM2}, 
\begin{align*}
0 = \int_{\Omega_h} (\nab \cdot \bu_h) q = 
\sum_{T\in \mct} \int_T (\nab \cdot \bu_h)q 
&= \sum_{T\in \mct} \frac1{2|\tilde T|} \int_{\hat T} \frac{\hat \nab \cdot \hat \bu_T}{\det(DF_T)} (\hat \nab \cdot \hat \bu_T)\det(DF_T)\\
& = \sum_{T\in \mct} \frac1{2|\tilde T|} \int_{\hat T} |\hat \nab \cdot \hat \bu_T|^2.
\end{align*}
Thus, $\hat \nab \cdot \hat \bu_T = 0$ for all $T\in \mct$, and therefore $\nab \cdot \bu_h=0$. 
\end{proof}

\subsection{Convergence Analysis}

Define the subspace of divergence--free functions:
\[
\bX^h:=\{\bv\in \bV^h:\ \nab \cdot \bv = 0\}\not \subset \bX:=\{\bv\in \bH^1_0(\Omega):\ \nab \cdot \bv=0\}.
\]
Then by Lemma \ref{lem:DivFree} and \eqref{eqn:FEM}, the discrete velocity solution is uniquely determined by
the problem: Find $\bu_h\in \bX^h$ such that
\begin{align*}
a_h(\bu_h,\bv):=\int_{\Omega_h} \nu \nab \bu_h: \nab \bv & = \int_{\Omega_h} {\bm f}_h\cdot \bv\qquad \forall \bv\in\bX^h.
\end{align*}
Standard theory of non--conforming and mixed finite element methods (e.g., \cite{Brenner}), along with Lemma \ref{lem:Interp} 
and Theorem \ref{thm:GlobalInfsup} yields 
\begin{align}\label{eqn:Strang}
\nu \|\nab (\bu-\bu_h)\|_{L^2(\Omega_h)}
&\le \inf_{\bw\in \bX^h} \nu \|\nab (\bu-\bw)\|_{L^2(\Omega_h)}+\sup_{\bv\in \bX^h\backslash \{0\}} \frac{a_h(\bu_h-\bu,\bv)}{\|\nab \bv\|_{L^2(\Omega_h)}}\\
&\nonumber\le C \inf_{\bw\in \bV^h} \nu \|\nab (\bu-\bw)\|_{L^2(\Omega_h)}+\sup_{\bv\in \bX^h\backslash \{0\}} \frac{a_h(\bu_h-\bu,\bv)}{\|\nab \bv\|_{L^2(\Omega_h)}}\\
%
&\nonumber\le C h^2 \nu \|\bu\|_{H^3(\Omega)}+\sup_{\bv\in \bX^h\backslash\{0\}} \frac{a_h(\bu_h-\bu,\bv)}{\|\nab \bv\|_{L^2(\Omega_h)}}.
\end{align}

Recalling  $-\nu \Delta \bu+\nab p = {\bm f}$ in $\bbR^2$ and $\bX^h\subset \bH_0({\rm div};\Omega)$, we have
\begin{align*}
a_h(\bu_h-\bu,\bv) 
%
&= \int_{\Omega_h} {\bm f}\cdot \bv - a_h(\bu,\bv)+\int_{\Omega_h} ({\bm f}_h-{\bm f})\cdot \bv\\
& = -\int_{\Omega_h} \nu \Delta \bu\cdot \bv + \int_{\Omega_h} \nab p\cdot \bv - a_h(\bu,\bv)+\int_{\Omega_h} ({\bm f}_h-{\bm f})\cdot \bv\\
& = -\int_{\Omega_h} \nu \Delta \bu\cdot \bv  - a_h(\bu,\bv)+\int_{\Omega_h} ({\bm f}_h-{\bm f})\cdot \bv\qquad \forall \bv\in \bX^h.
\end{align*}

We then integrate by parts to conclude 
\begin{align*}
-\int_{\Omega_h} \nu \Delta \bu\cdot \bv  - a_h(\bu,\bv)
 = \nu \sum_{e\in \mathcal{E}_h^I} \int_e \nab \bu :[\bv].
\end{align*}
where we used  $\bv$ is zero on $\p\Omega_h$.

Recall $\mathcal{E}_h^{I,\p}$ is the set of edges
in $\mathcal{E}_h^I$ that have one endpoint on $\p\Omega_h$.
Then by properties of $\bV^h$, there holds $[\bv]|_e=0$ for all $e\in \mathcal{E}^I_h\backslash \mathcal{E}_h^{I,\p}$, in particular,
\begin{align}\label{eqn:Strang22}
a_h(\bu_h-\bu,\bv)
& =  \nu \sum_{e\in \mathcal{E}_h^{I,\p}} \int_e \nab \bu :[\bv] +\int_{\Omega_h} ({\bm f}_h-{\bm f})\cdot \bv.
\end{align}

\begin{lemma}\label{lem:ConsEst}
There holds
\[
  \nu \sum_{e\in \mathcal{E}_h^{I,\p}} \int_e \nab \bu :[\bv]\le C \nu  h^2 \|\bu\|_{H^3(\Omega)}\|\nab \bv\|_{L^2(\Omega_h)}\qquad \forall \bv\in \bV^h.
\]
\end{lemma}
\begin{proof}
For each $e\in \mathcal{E}_h^{I,\p}$, let $G_e\in \bbR^{2\times 2}$ be the average of $\nab \bu$ on $e$. 
Standard interpolation estimates show 
\begin{align}\label{eqn:GeBound}
h_e^{-1} \|\nab \bu- G_e\|_{L^2(e)}^2\le C  |\bu|_{H^2(T)}\qquad h_e = {\rm diam}(e),
\end{align}
for $T$ satisfying $e\subset \p T$.  Moreover, we clearly have $|G_e|\le C |\bu|_{W^{1,\infty}(\Omega)}.$

Let $\bE_h \bv\in \bH^1_0(\Omega_h)$ satisfy \eqref{eqn:EhApprox}.
Then $[\bE_h \bv]|_e=0$ for all $e\in \mathcal{E}_h^I$, and so by \eqref{eqn:Strang22},
\begin{align}
\label{eqn:Strang33}
 \nu \sum_{e\in \mathcal{E}_h^{I,\p}} \int_e \nab \bu :[\bv]
%
& = \nu \sum_{e\in \mathcal{E}_h^{I,\p}} \Big(\int_e (\nab \bu-G_e): [\bv-\bE_h \bv] + \int_e G_e:[\bv]\Big)\\ 
&\nonumber =: I_1+I_2.
\end{align}

To bound $I_1$ we use the Cauchy--Schwarz inequality, \eqref{eqn:GeBound}, a trace inequality, and Lemma \ref{lem:EhApprox}:
\begin{align}\label{eqn:I1bound}
I_1 
&\le \nu  \Big(\sum_{e\in \mathcal{E}_h^{I,\p}} h_e^{-1}\|\nab \bu-G_e\|_{L^2(e)}^2\Big)^{1/2}
\Big(\sum_{e\in \mathcal{E}_h^{I,\p}} h_e \|[\bv-\bE_h\bv]\|_{L^2(e)}^2\Big)^{1/2}\\
&\nonumber\le C \nu  h^2 |\bu|_{H^2(\Omega)} \|\nab \bv\|_{L^2(\Omega_h)}.
\end{align}

To bound $I_2$, we apply Lemma \ref{lem:JumpEst}.
\begin{align}\label{eqn:I2bound}
I_2  
&=  \nu \sum_{e\in \mathcal{E}_h^{I,\p}} G_e :\int_e [\bv]\\
%
%
%
&\nonumber\le C \nu |\bu|_{W^{1,\infty}(\Omega)}\Big(\sum_{e\in \mathcal{E}_h^{I,\p}} h_e\Big)^{1/2}\Big(\sum_{e\in \mathcal{E}_h^{I,\p}}h_e^{-1}\Big|\int_e [\bv]\Big|^2\Big)^{1/2}\\
&\nonumber\le C \nu h^{5/2} \|\bu\|_{W^{1,\infty}(\Omega)} \Big( \sum_{e\in \mathcal{E}_h^{I,\p}} h_e\Big)^{1/2} \|\nab \bv\|_{L^2(\Omega_h)}
\le C \nu h^{5/2} \|\bu\|_{H^{3}(\Omega)} \|\nab \bv\|_{L^2(\Omega_h)}.
\end{align}
Combining \eqref{eqn:Strang33}--\eqref{eqn:I2bound} yields the desired result.
\end{proof}

Finally, we combine \eqref{eqn:Strang}, \eqref{eqn:Strang22},  Lemma \ref{lem:ConsEst} to obtain
the main result of the section.
\begin{theorem}\label{thm:Main}
There holds
\begin{align}\label{eqn:MainVelocityEstimate}
\|\nab (\bu- \bu_h)\|_{L^2(\Omega_h)}
&\le C\big( h^2 \|\bu\|_{H^3(\Omega)} + \nu^{-1} |{\bm f}-{\bm f}_h|_{X_h^*}\big),
\end{align}
where 
\[
|{\bm f}-{\bm f}_h|_{X_h^*} = \sup_{\bv\in \bX_h\backslash \{0\}} \frac{\int_{\Omega_h} ({\bm f}-{\bm f}_h)\cdot \bv}{\|\nab \bv\|_{L^2(\Omega_h)}}.
\]
Therefore if, for example, ${\bm f}_h$ is the nodal quadratic interpolant of ${\bm f}$, and if ${\bm f}$ is sufficiently smooth, there holds
\begin{align*}
\|\nab (\bu- \bu_h)\|_{L^2(\Omega_h)}
&\le C\big( h^2 \|\bu\|_{H^3(\Omega)} + \nu^{-1} h^3 \|{\bm f}\|_{H^3(\Omega)}\big).
\end{align*}

The pressure approximation satisfies
\begin{align}\label{eqn:PressureError}
\|p-p_h\|_{L^2(\Omega_h)}\le C\big( \nu \|\nab (\bu-\bu_h)\|_{L^2(\Omega_h)} + \nu h^2 \|\bu\|_{H^3(\Omega)}+ \inf_{q\in Q^h} \|p-q\|_{L^2(\Omega_h)} + \|{\bm f}-{\bm f}_h\|_{L^2(\Omega_h)}\big).
\end{align}
\end{theorem}
\begin{proof}
The error estimate of the velocity follows from  \eqref{eqn:Strang}, \eqref{eqn:Strang22}, and Lemma \ref{lem:ConsEst}; 
thus, it remains to 
prove \eqref{eqn:PressureError}.

For any $q\in Q^h$ and $\bv\in \bV^h$, we have by \eqref{eqn:FEM} and Lemma \ref{lem:ConsEst},
\begin{align*}
&\int_{\Omega_h} (\nab \cdot \bv_h)(p_h-q) 
= a_h(\bu_h,\bv) -\int_{\Omega_h} (\nab \cdot \bv)q -\int_{\Omega_h} {\bm f}_h \cdot \bv\\
%
%
&\qquad = a_h(\bu_h-\bu,\bv)  -\int_{\Omega_h} (\nab \cdot \bv)(q-p)- \int_{\Omega_h} ({\bm f}_h-{\bm f})\cdot \bv + \nu \sum_{e\in \calE_h^{I,\p}}  \int_e \nab \bu:[\bv]\\
&\qquad \le C \big(\nu \|\nab (\bu-\bu_h)\|_{L^2(\Omega_h)} +  \nu h^2 \|\bu\|_{H^3(\Omega)}+   \|p-q\|_{L^2(\Omega_h)}\big)\|\nab \bv\|_{L^2(\Omega_h)} + \|{\bm f}-{\bm f}_h\|_{L^2(\Omega_h)}\|\bv\|_{L^2(\Omega_h)}.
\end{align*}
Using the estimate \eqref{eqn:EhApprox} and the Poincare inequality, we have 
\[
\|\bv\|_{L^2(\Omega_h)}\le \|{\bm E}_h \bv\|_{L^2(\Omega_h)}+ \|\bv-\bE_h \bv\|_{L^2(\Omega_h)}\le C \|\nab \bv\|_{L^2(\Omega_h)}.
\]
Therefore
\begin{align*}
\int_{\Omega_h} (\nab \cdot \bv_h)(p_h-q)
&\le C \big(\nu \|\nab (\bu-\bu_h)\|_{L^2(\Omega_h)} + \nu h^2 \|\bu\|_{H^3(\Omega)}\\
&\qquad +  \|p-q\|_{L^2(\Omega_h)} + \|{\bm f}-{\bm f}_h\|_{L^2(\Omega_h)}\big)\|\nab \bv\|_{L^2(\Omega_h)}.
\end{align*}
We then use the inf-sup condition given in Theorem \ref{thm:GlobalInfsup} to obtain
\begin{align*}
 C \|p_h-q\|_{L^2(\Omega_h)} 
&\le \sup_{\bv\in \bV^h\backslash \{0\}} \frac{\int_{\Omega_h} (\nab \cdot \bv)(p_h-q)}{\|\nab \bv\|_{L^2(\Omega_h)}} \\
&\le C \big(\nu \|\nab (\bu-\bu_h)\|_{L^2(\Omega_h)} +  \nu h^2 \|\bu\|_{H^3(\Omega)}+ \|p-q\|_{L^2(\Omega_h)} + \|{\bm f}-{\bm f}_h\|_{L^2(\Omega_h)}\big).
\end{align*}
Applying the triangle inequality and taking the infimum over $q\in Q^h$, we obtain \eqref{eqn:PressureError}.
\end{proof}

\section{A Pressure Robust Scheme}\label{sec:Robust}
In this section, we construct a computable approximation ${\bm f}_h$
such that the term $\nu^{-1} |{\bm f}-{\bm f}_h|_{\bX_h^*}$ appearing in estimate \eqref{eqn:MainVelocityEstimate}
is independent of the viscosity, in particular, such that the method is pressure robust.
Essentially, this construction
is done by applying a commuting operator to the function ${\bm f}|_{\Omega}$.
In particular, we adopt and modify the recent results in \cite{FuGuzmanNeilan2020}
for Scott--Vogelius elements to construct commuting operators
on meshes with curved boundary.

To discuss the main objections of this section further, we define
the rot operator
\[
{\rm rot}\,\bv = \frac{\p v_2}{\p x_1} - \frac{\p v_1}{\p x_2},
\]
and the corresponding Hilbert space
\begin{align*}
\bH({\rm rot};\Omega_h):=\{\bv\in \bL^2(\Omega_h):\ {\rm rot}\,\bv\in L^2(\Omega_h)\}.
\end{align*}

The main goal of this section is to prove the following result.
\begin{theorem}\label{thm:Commute}
There exists finite element spaces $\bW_h\subset \bH({\rm rot};\Omega_h)$,
 $\Sigma_h\subset H^1_0(\Omega)$ with respect to the partition
$\mct$, and operators $\bPi_W:\bH^2(\Omega)\to \bW_h$
and $\Pi_{\Sigma}:H^3(\Omega)\to \Sigma_h$ such that
\begin{equation}\label{eqn:CommutingP}
\bPi_{W} \nab p = \nab \Pi_\Sigma p\qquad \forall p\in H^3(\Omega).
\end{equation}
Moreover, there holds for any ${\bm f}\in H^3(\Omega)$,
\begin{equation}\label{eqn:PiWEst}
\|{\bm f}-\bPi_W {\bm f}\|_{L^2(\Omega_h)}\le C h^2 \|{\bm f}\|_{H^3(\Omega)},
\end{equation}
where ${\bm f}$ in the left-hand side of the above inequality is an $H^3$ extension of ${\bm f}|_{\Omega}$.
\end{theorem}

\begin{corollary}
Let $(\bu_h,p_h)\in \bV_h\times Q_h$ be
the solution of the finite element method \eqref{eqn:FEM}
with ${\bm f}_h = \bPi_W {\bm f}$.  Then there holds
\begin{align*}
\|\nab (\bu-\bu_h)\|_{L^2(\Omega_h)}\le C h^2 \|\bu\|_{H^5(\Omega)}.
\end{align*}
\end{corollary}
\begin{proof}
In light of estimate \eqref{eqn:MainVelocityEstimate}, it suffices to show $|{\bm f}-{\bm f}_h|_{X_h^*}\le C \nu h^2 \|\bu\|_{H^5(\Omega)}$.

Recall that the extension of ${\bm f}|_{\Omega}$ is given by
${\bm f} = -\nu \Delta \bu + \nab p$.  Therefore by Theorem \ref{thm:Commute}, for all $\bv\in \bX_h$,
\begin{align*}
\int_{\Omega_h} ({\bm f}-{\bm f}_h)\cdot \bv 
& = \int_{\Omega_h} \big(-\nu (\Delta \bu-\bPi_W \Delta \bu) + (\nab p - \bPi_W \nab p)\big)\cdot \bv\\
& = \int_{\Omega_h} \big(-\nu (\Delta \bu-\bPi_W \Delta \bu) + \nab ( p -  \Pi_{\Sigma} p)\big)\cdot \bv\\
& = -\nu \int_{\Omega_h} (\Delta \bu-\bPi_W \Delta \bu) \cdot \bv,
\end{align*}
where used that $\nab \cdot \bv=0$ and $\bv\cdot \bn|_{\p\Omega_h}=0$.
Consequently,
\begin{align*}
|{\bm f}-{\bm f}_h|_{X_h^*}\le C \nu \|\Delta \bu- \bPi_W \Delta \bu\|_{L^2(\Omega_h)}\le C h^2 \nu  \|\Delta \bu\|_{H^3(\Omega)}\le C \nu h^2 \|\bu\|_{H^5(\Omega)}.
\end{align*}\hfill
\end{proof}

\subsection{Proof of Theorem \ref{thm:Commute}: Preliminaries}

As a first step of the proof of Theorem \ref{thm:Commute}, we ``rotate'' the space $\bV(T)$.
\begin{definition}
We define
\begin{align*}
\bW(T)& = \{\bv\in \bH^1(T):\ \bv(x) = (DF_T(\hat x))^{-\intercal} \hat \bv(\hat x),\ \exists \hat \bv\in \hat \bV\},\\
\bW_0(T)& = \bW(T)\cap \bH^1_0(T).
\end{align*}
\end{definition}

\begin{remark}
Define 
\[
S = \begin{pmatrix}
0 & -1\\
1 & 0
\end{pmatrix},
\]
so that
${\rm rot}(S\bv) = \nab \cdot \bv$, and $S DF_TS^{-1} = \det(DF_T) (DF_T)^{-\intercal}$. 
Therefore, if $\bv(x) =  (DF_T(\hat x))^{-\intercal} \hat \bv(\hat x)$, we have
\begin{align*}
{\rm rot}\, \bv(x)
& = {\rm rot}\Big(S \frac{DF_T(\hat x)S^{-1}\hat \bv(\hat x)}{\det(DF_T(\hat x))}\Big)
 = \nab \cdot \Big( \frac{DF_T(\hat x)S^{-1}\hat \bv(\hat x)}{\det(DF_T(\hat x))}\Big)
%
%
 = \frac{\hat {\rm rot}\,\hat \bv(\hat x)}{\det(DF_T(\hat x))}.
\end{align*}
\end{remark}

\begin{remark}
Note that $\hat {\rm rot}:\hat \bV\to \hat Q$ is a surjection.
Indeed, let $\hat q\in \hat Q$. 
Then there exists $\hat \bv\in \bV$ such that $\hat \nab \cdot \hat \bv = \hat q$.
Then set $\hat \bw = S \hat \bv$ so that $\hat q = \hat \nab \cdot \hat \bv = \hat {\rm rot}\,\hat \bw$.
Similar arguments show  $\hat {\rm rot}:\hat \bV_0\to \hat Q_0$ is a bijection.
\end{remark}

\begin{lemma}\label{lem:WDOFs}
Let $\{\hat \alpha_i\}_{i=1}^3,\{\hat m_i\}_{i=1}^3 \subset \mathcal{N}_{\hat T}$
be, respectively, the vertices and edge midpoints of $\hat T$.
Set $\alpha_i = F_T(\hat \alpha_i)$ and $m_i = F_T(\hat m_i)$ to be the corresponding points on $T$.
Any $\bv\in \bW(T)$ is uniquely determined by the values
\begin{subequations}
\label{eqn:WDOFS}
\begin{alignat}{2}
&\bv(\alpha_i),\ (\bv\cdot \bn)(m_i)\qquad &&i=1,2,3,\\
&\int_e \bv\cdot \bt\qquad &&\text{$\forall$ edges of $T$},\\
&\int_T ({\rm rot}\,\bv)q\qquad &&\forall q\in {Q}_0(T).
\end{alignat}
\end{subequations}
\end{lemma}
\begin{proof}
Write $\bv(x) = DF_T^{-\intercal} \hat \bv$ for some $\hat \bv\in \hat \bV$,
and suppose that $\bv$ vanishes on the DOFs.
We show $\hat \bv\equiv 0$.

We clearly have $\hat \bv(\hat \alpha_i) = 0$ for $i=1,2,3$,
and by using the relation
$\bt = {DF_T\hat \bt}/{|DF_T \hat \bt|}$ \cite{MonkBook},
and a change of variables, we have
\begin{align*}
0 = \int_e \bv\cdot \bt = \int_{\hat e} \frac{(DF_T^{-\intercal} \hat \bv)\cdot (DF_T\hat \bt)}{|DF_T\hat \bt|}
|\det(DF_T)| |DF^{-\intercal} \hat \bn| = \int_{\hat e} \hat \bv\cdot \hat \bt,
\end{align*}
where we used the identity $|\det(DF_T)| |DF^{-\intercal} \hat \bn| = |DF_T \hat \bt|$.
Thus, we conclude  $\hat \bv\cdot \hat \bt|_{\p \hat T} = 0$.

Similarly, using the relation
$\bn = {DF_T^{-\intercal}\hat \bn}/{|DF_T^{-\intercal} \hat \bn|}$,
we compute
\begin{align*}
0 &= (\bv\cdot \bn)(m_i)
%
 = \frac{\hat \bv \cdot (DF_T^{-1} DF_T^{-\intercal} \hat \bn)}{|DF_T^{-\intercal} \hat \bn|}(\hat m_i).
\end{align*}
Because $(DF_T^{-1} DF_T^{-\intercal} \hat \bn)\cdot \hat \bn = |DF_T^{-\intercal} \hat \bn|^2 \neq 0$, 
we conclude  $(DF_T^{-1} DF_T^{-\intercal} \hat \bn)$ is not
tangent to $\hat \bt$. Thus, since $\hat \bv\cdot \hat \bt|_{\p \hat T}=0$, 
we get  $\hat \bv|_{\p \hat T}=0$, i.e., $\hat \bv\in \hat \bV_0$.

Now let $\hat q\in \hat Q_0$, and set $q(x) = \hat q(\hat x)$ so that $q\in Q_0(T)$.
Using ${\rm rot}\,\bv = {\hat {\rm rot}\, \hat \bv}/{\det(DF_T)}$, we have by a change of variables,
\[
\int_T ({\rm rot}\, \bv) q = \int_{\hat T} (\hat {\rm rot}\, \hat \bv) \hat q.
\]
Taking $\hat q = \hat {\rm rot} \, \hat \bv$, we conclude $\hat {\rm rot}\, \hat \bv=0$.
This implies $\hat \bv\equiv 0$, and therefore $\bv\equiv 0$.
\end{proof}

Next, we define the local Clough-Tocher space on the reference element
\[
\hat \Sigma = \{\hat \sigma \in H^2(\hat T):\ \hat \sigma|_{\hat K}\in \pol_3(\hat K)\ \forall \hat K\in \hat T^{ct}\}.
\]
It is known that the dimension of $\hat \Sigma$ is $12$ \cite{Ciarlet}, and any $\hat \sigma\in \hat \Sigma$
is uniquely determined by the valuess
\begin{alignat}{2}\label{eqn:ShatDOFs}
&\hat \nabla \hat \sigma (\hat \alpha_i),\ \hat \sigma(\hat \alpha_i),\ (\hat \nab \hat \sigma\cdot \hat \bn)(\hat m_i)\qquad &&i=1,2,3.%
\end{alignat}

We define the Clough--Tocher space on $T$ via composition
\[
\Sigma(T) = \{\sigma:\ \sigma(x) = \hat \sigma(\hat x),\ \exists \hat \sigma\in \hat \Sigma\}.
\]
It is easy to see $\Sigma(T) \subset H^2(T)$.
In the following lemma, we extend
the above DOFs to $\Sigma(T)$.
\begin{lemma}\label{lem:SDOFs}
A function $\sigma\in \Sigma(T)$ is uniquely determined by the values
\begin{subequations}
\label{eqn:SDOFsS}
\begin{alignat}{2}
& \nabla  \sigma (\alpha_i),\  \sigma(\alpha_i),\ (\nab  \sigma\cdot  \bn)(m_i)\qquad &&i=1,2,3.
%
\end{alignat}
\end{subequations}
\end{lemma}
\begin{proof}
Write $\sigma(x) = \hat \sigma(\hat x)$ with $\hat \sigma\in \hat \Sigma$.
It suffices to show that  if $\sigma$ vanishes at the above DOFs, 
then $\hat \sigma$ vanishes on \eqref{eqn:ShatDOFs}.

If $\sigma$ vanishes at the above DOFs, then clearly
\begin{align*}
\hat \nabla \hat \sigma (\hat \alpha_i) = 0,\ \hat \sigma(\hat \alpha_i) = 0\qquad i=1,2,3.
\end{align*}
This implies  $\hat \sigma|_{\p \hat T}=0$, and therefore $\hat \nab \hat \sigma\cdot \hat \bt|_{\p \hat T} = 0$.

Next, by the chain rule and the relation $\bn = {DF^{-\intercal} \hat \bn}/{|DF^{-\intercal}\hat \bn|}$,
\begin{align*}
0= (\nab \sigma\cdot \bn)(m_i) 
%
 = \Big(\frac1{|DF^{-\intercal}_T \hat \bn|}  \hat \nab \hat \sigma \cdot (DF^{-1}_T DF^{-\intercal}_T  \hat \bn)\Big)(\hat m_i).
\end{align*}
Thus, we have $\big(\hat \nab \hat \sigma \cdot (DF_T^{-1} DF^{-\intercal}_T  \hat \bn)\big)(\hat m_i)=0$. 
Since
\begin{align*}
\big((DF_T^{-1} DF_T^{-\intercal}  \hat \bn) \cdot \hat \bn\big)(\hat m_i)
& = |(DF_T\hat \bn)(\hat m_i)|^2 \neq 0,
\end{align*}
 the vector $(DF_T^{-1} DF_T^{-\intercal}  \hat \bn)(\hat m_i)$ is not tangent to $\hat e$.
Because the tangental derivative of $\hat \sigma$ vanishes at $\hat m_i$, 
we  conclude  $\hat \nab \hat \sigma(\hat m_i)=0$.
Thus, $\hat \sigma\equiv 0$ and $\sigma \equiv 0$.
\end{proof}


\begin{remark}
Note that if $\sigma\in \Sigma(T)$ with $\sigma(x) = \hat \sigma(\hat x)$, then
$\nab \sigma(x) = (DF_T(\hat x))^{-\intercal}\hat \nab \hat \sigma(\hat x)$.
We conclude  $\nab \sigma\in \bW(T)$.
\end{remark}

As a next step, we
use the DOFs stated in Lemmas \ref{lem:WDOFs}--\ref{lem:SDOFs} to construct commuting operators with properties 
stated in Theorem \ref{thm:Commute}.
Note that an added difficulty of the construction is that the operators
are defined for functions with domain $\Omega$,
but map to functions with domain $\Omega_h$.
To mitigate this mismatch, we employ
the mapping $G:\tilde \Omega_h\to \Omega$
given in Section \ref{sec:Prelim}.

For each $T\in \mct$ and edge $e$ in $\mct$, we set
\[
T_R:=G(G_h^{-1}(T))\subset \Omega,\qquad e_R := G(G_h^{-1}(e))\subset \bar\Omega,
\]
where we recall $G_h$ is the quadratic interpolant of $G$.
That is, $T_R$ is obtained by first mapping
$T$ to its associated affine element $\tilde T = G_h^{-1}(T)\in \tilde \mct$, 
and then mapping $\tilde T$ to $G(\tilde T)\subset \Omega$.
By properties of the quadratic interpolant $G_h$, we have $G(G_h^{-1}(\alpha_i)) = \alpha_i$
and $G(G_h^{-1}(m_i)) = m_i$ for all vertices and edge midpoints of $T$.


Via Lemmas \ref{lem:WDOFs}--\ref{lem:SDOFs} 
we introduce the operator
$\bPi_W^T:\bH^2(T_R)\to \bW(T)$ 
uniquely determined by the conditions
\begin{subequations}
\label{eqn:PiW}
\begin{alignat}{2}
\label{eqn:PiW1}
&(\bPi^T_W \bv)(\alpha_i) = \bv(\alpha_i)\qquad &&i=1,2,3,\\
\label{eqn:PiW2}
&(\bPi^T_W \bv \cdot \bn)(m_i) = (\bv\cdot \bn)(m_i)\qquad &&i=1,2,3,\\
\label{eqn:PiW3}
&\int_e (\bPi^T_W \bv)\cdot \bt = \int_{e_R} \bv\cdot \bt_{e_R}\qquad &&\text{$\forall$ edges of $T$},\\
\label{eqn:PiW4}
&\int_{T} ({\rm rot}\,\bPi^T_W \bv)q = \int_{T\cap T_R} ({\rm rot}\,\bv)q\qquad &&\forall q\in {Q}_0(T),
\end{alignat}
\end{subequations}
where $\bn$ is the outward unit normal with respect to $e\subset \p T$,
$\bt$ is the unit tangent of $e\subset \p T$, and $\bt_{e_R}$ is the unit tangent
of $e_R\subset \p T_R$.
We also set  $\Pi_\Sigma^T:H^3(T_S)\to \Sigma(T)$ uniquely determined by
\begin{subequations}
\label{eqn:PiS}
\begin{alignat}{2}
\label{eqn:PiS1}
&\Pi^T_\Sigma \sigma(\alpha_i) = \sigma(\alpha_i),\quad \nabla  (\Pi_{\Sigma} \sigma) (\alpha_i) = \nab \sigma(a),\  \qquad &&i=1,2,3,\\
\label{eqn:PiS2}
& \nab  (\Pi^T_\Sigma\sigma)( m_i)\cdot  \bn(m_i) =  \nab  \sigma( m_i)\cdot  \bn(m_i) \qquad &&i=1,2,3.
\end{alignat}
\end{subequations}

We define the global spaces
\begin{align*}
\bW^h &= \{\bv\in {\bm H}({\rm rot};\Omega_h):\ \bv|_T\in \bW(T)\ \forall T\in \mct,\text{ $\bv$ is continuous on \eqref{eqn:WDOFS}}\},\\
\Sigma^h &= \{\sigma\in H^1(\Omega_h):\ \sigma|_T\in \Sigma(T)\ \forall T\in \mct,\text{ $\sigma$ is continuous on \eqref{eqn:SDOFsS}}\},
\end{align*}
and the operators $\bPi_W:H^2(\Omega)\to \bW^h$, $\Pi_\Sigma:H^3(\Omega)\to \Sigma^h$ by
\[
\bPi_W \bv|_T = \bPi^T_W \bv,\qquad \Pi_\Sigma \sigma|_T = \Pi^T_\Sigma \sigma,\qquad \forall T\in \mct.
\]
We now prove that these operators satisfy \eqref{eqn:CommutingP}--\eqref{eqn:PiWEst}.

\subsection{Proof of \eqref{eqn:CommutingP}}
For given $p\in H^3(\Omega)$,
set $\brho = \bPi_W \nab  p - \nab \Pi_\Sigma p \in \bW(T)$.
We wish to show  $\brho \equiv 0$.
This this end, it suffices to show  $\brho$
vanishes at the DOFs in Lemma \ref{lem:WDOFs} for each $T\in \mct$.

First, we consider the interior DOFs of $\bW(T)$.  Using \eqref{eqn:PiW4} and the identity ${\rm rot}\,\nab p = 0$, we have
\begin{align*}
\int_T ({\rm rot}\,\brho)q
= \int_T  ({\rm rot}\,(\bPi_W \nab  p ))q
= \int_{T\cap T_R}  ({\rm rot}\,( \nab  p ))q
=0\qquad \forall q\in Q_0(T).
\end{align*}

Let $\alpha_i$ be a vertex of $T$.  We then have by \eqref{eqn:PiW1} and \eqref{eqn:PiS1},
\begin{align*}
\brho(\alpha_i) = \bPi_W \nab  p(\alpha_i) - \nab \Pi_\Sigma p(\alpha_i) = 0.
\end{align*}
Next, let $m_i$ be an edge midpoint of $T$ and let $\bn$ be the outward unit
normal at $m_i$.  Then by \eqref{eqn:PiW2} and \eqref{eqn:PiS2},
\begin{align*}
\brho(m_i)\cdot \bn = \bPi_W \nab  p(m_i)\cdot \bn - \nab \Pi_\Sigma p(m_i)\cdot \bn = 0.
\end{align*}

Finally, let $e\subset \p T$ be an edge of $T$ with
endpoints $\alpha_2$ and $\alpha_1$.  Recalling that $e_R$ also has endpoints
$\alpha_2$ and $\alpha_1$, we use \eqref{eqn:PiS1} and \eqref{eqn:PiW3} to obtain
\begin{align*}
\int_e \brho\cdot \bt
 = \int_e \big(\bPi_W \nab p - \nab \Pi_{\Sigma} p\big)\cdot \bt
& = \int_{e_R} \nab p \cdot \bt_{e_R} - \int_e (\nab \Pi_{\Sigma} p)\cdot \bt\\
& = p(\alpha_2)-p(\alpha_1) - \big((\Pi_{\Sigma}p)(\alpha_2)-(\Pi_{\Sigma}p)(\alpha_1)\big) = 0.
\end{align*}
Thus, $\brho$ vanishes at all the DOFs in Lemma \ref{lem:WDOFs},
and we conclude $\brho\equiv 0$.

\subsection{Proof of \eqref{eqn:PiWEst}}
We break up the proof of  estimate \eqref{eqn:PiWEst} into four parts.

\begin{enumerate}
\item[(i)] We  extend ${\bm f}$ to $\bbR^2$ such that $\|{\bm f}\|_{H^3(\mathbb{R})}\le C\|{\bm f}\|_{H^3(\Omega)}$
With this extension, we define $\bI_W {\bm f}\in \bW(T)$ uniquely
by the conditions
\begin{alignat*}{2}
&(\bI_W {\bm f})(\alpha_i) = {\bm f}(\alpha_i),\quad (\bI_W {\bm f} \cdot \bn)(m_i) = ({\bm f}\cdot \bn)(m_i)\qquad &&i=1,2,3,\\
%
%
&\int_e (\bI_W {\bm f})\cdot \bt = \int_{e} {\bm f}\cdot \bt\qquad &&\text{$\forall$ edges of $T$},\\
&\int_{T} ({\rm rot}\,\bI_W {\bm f})q = \int_{T} ({\rm rot}\,{\bm f})q\qquad &&\forall q\in {Q}_0(T).
\end{alignat*}
\item[(ii)]
We now estimate $\|{\bm f}-\bI_W {\bm f}\|_{L^2(T)}$.
For notational convenience, we write $\bv = \bI_W {\bm f}$, and set 
\[
\bv(x) = R_T(\hat x) \hat \bv(\hat x),\qquad {\bm f}(x) = R_T(\hat x) \hat{\bm f}(\hat x),
\]
with $R_T(\hat x) = (DF_T(\hat x))^{-\intercal}$.
We then have 
\begin{align}\label{eqn:IW1}
&\hat \bv(\hat \alpha_i) = \hat{\bm f}(\hat \alpha_i),\quad (\hat \bv\cdot (R_T^\intercal \bn))(\hat m_i) = (\hat {\bm f}\cdot (R_T^{\intercal}\bn))(\hat m_i)\qquad i=1,2,3.
\end{align}
We also have, by a change of variables (cf.~proof of Lemma \ref{lem:WDOFs})
\begin{align}\label{eqn:IW2}
\int_{\hat e} \hat \bv\cdot \hat \bt = \int_{e} \bv\cdot \bt = \int_e {\bm f}\cdot \bt = \int_{\hat e} \hat{\bm f}\cdot \hat \bt.
\end{align}

Next, for $q\in Q_0(T)$, write $q(x) = \hat q(\hat x)$ with $\hat q\in \hat Q$.
We then have
\begin{align}\label{eqn:IW3}
\int_{\hat T} (\hat {\rm rot} \hat \bv)\hat q
& = \int_{\hat T} ( \det(DF_T){\rm rot}  \bv)\circ F_T \hat q = \int_T {\rm rot} \bv q = \int_{T} {\rm rot} {\bm f} q = \int_{\hat T} \hat {\rm rot} \hat {\bm f} \hat q
\end{align}
{It follows from \eqref{eqn:IW1}--\eqref{eqn:IW3} 
and a slight generalization of the Bramble--Hilbert lemma that
\begin{equation}\label{eqn:BHLem}
\|\hat {\bm f}-\hat \bv\|_{L^2(\hat T)}\le C|\hat {\bm f}|_{H^3(\hat T)}.
\end{equation}
}
Therefore by Lemma \ref{lem:scaling} and \eqref{eqn:BHLem} (and noting $R_T^{-1} = DF_T^\intercal$),
\begin{align}\label{eqn:PiFirst}
\|{\bm f}-\bI_W {\bm f}\|_{L^2(T)}
&\le C h_T \|R_T (\hat {\bm f}-\hat \bv)\|_{L^2(\hat T)}\\
%
%
&\nonumber\le C |\hat {\bm f}|_{H^3(\hat T)} = C  |R_T^{-1} R_T \hat {\bm f}|_{H^3(\hat T)}\\
&\nonumber\le C  \big(\|R_T^{-1}\|_{L^\infty(\hat T)} |R_T \hat {\bm f}|_{H^3(\hat T)}+|R_T^{-1}|_{W^{1,\infty}(\hat T)} |R_T \hat {\bm f}|_{H^2(\hat T)}\big)\\
%
%
&\nonumber\le C h_T^3 \|{\bm f}\|_{H^3(T)}.
\end{align}


\item[(iii)]
We now  estimate $(\bPi_W {\bm f}-\bI_W {\bm f})|_T\in \bW(T)$.
Set $\bw = \bPi_W {\bm f}-\bI_W {\bm f}\in \bW(T)$.  Then
\begin{alignat*}{2}
&\bw(\alpha_i) = 0,\quad (\bw \cdot \bn)(m_i) = 0\qquad &&i=1,2,3,\\
%
%
&\int_e \bw\cdot \bt = \int_{e_R} {\bm f}\cdot \bt_{e_R} -  \int_{e} {\bm f}\cdot \bt\qquad &&\text{$\forall$ edges of $T$},\\
&\int_{T} ({\rm rot}\, \bw) q = \int_{T\cap T_R} ({\rm rot}\, {\bm f}) q - \int_{T} ({\rm rot}\,{\bm f})q\qquad &&\forall q\in {Q}_0(T).
\end{alignat*}

Write $\bw(x) = R_T(\hat x) \hat \bw(\hat x)$.
By scaling, we have
\begin{align}\label{eqn:ScalingXY}
\|\hat \bw\|_{H^m(\hat T)}^2
&\approx 
\sum_{i=1}^3 (|\hat \bw(\hat \alpha_i)|^2 + |\hat \bw(\hat m_i)|^2)+
\mathop{\sup_{\hat q\in \hat Q_0}}_{\|\hat q\|_{L^2(\hat T)}=1} \Big| \int_{\hat T} (\hat{\rm rot}\,\hat \bw)\hat q\Big|^2\\
&\nonumber=
\sum_{i=1}^3 |\hat \bw(\hat m_i)|^2+
\mathop{\sup_{\hat q\in \hat Q_0}}_{\|\hat q\|_{L^2(\hat T)}=1} \Big| \int_{\hat T} (\hat{\rm rot}\,\hat \bw)\hat q\Big|^2.
\end{align}

Next we use the algebraic identity
\begin{align}\label{eqn:AlgBasis}
\hat \bw(\hat m_i) =
\frac{1}{\balpha^\perp \cdot \bbeta} \Big( ((\hat \bw(\hat m_i))\cdot \bbeta )\balpha^\perp - ((\hat \bw(\hat m_i))\cdot \balpha) \bbeta^\perp\Big)
\end{align}
for any linearly independent vectors $\balpha,\bbeta\in \bbR^2$.  Here, $\balpha^\perp = S\balpha$.
We take $\balpha =-\hat \bt(\hat m_i)$ and $\bbeta = R^\intercal_T(\hat m_i) \bn(m_i)$, so that
\[
|\balpha^\perp \cdot \bbeta| = |S \hat \bt(\hat m_i) \cdot (R^\intercal_T(\hat m_i) \bn(m_i))| = |(R_T(\hat m_i) \hat \bn(\hat m_i))\cdot \bn(m_i)| =|(R_T\hat \bn)(\hat m_i)|,
\]
where we used the relation $\bn = R_T \hat \bn/|R_T \hat \bn|$ in the last equality.  We use
\eqref{eqn:AlgBasis} and the identity  
$\hat \bw(\hat m_i)\cdot \bbeta= (\bw \cdot \bn)(m_i) = 0$ to conclude
\begin{align*}
|\hat \bw(\hat m_i)| = \frac{1}{|R_T \hat \bn(\hat m_i)|} \big|(\hat \bw\cdot \hat \bt)(\hat m_i) S R_T^\intercal(\hat m_i) \bn(m_i)\big|
\le \frac{|R_T (\hat m_i)|}{|R_T \hat \bn(\hat m_i)|}  |(\hat \bw \cdot \bt)(\hat m_i)|\le C |(\hat \bw \cdot \bt)(\hat m_i)|,
\end{align*}
where $C>0$ is the condition number of $R_T$ (which is independent of $h$).

We use this estimate in \eqref{eqn:ScalingXY} to conclude
\begin{align*}
\|\hat \bw\|_{H^m(\hat T)}^2 \le C\Big( \sum_{i=1}^3 |(\hat \bw\cdot \hat \bt)(\hat m_i)|^2 + \mathop{\sup_{\hat q\in \hat Q_0}}_{\|\hat q\|_{L^2(\hat T)}=1} \Big| \int_{\hat T} (\hat{\rm rot}\,\hat \bw)\hat q\Big|^2\Big).
\end{align*}
Using Simpson's rule, noting that $\hat \bw$ vanishes on the vertices of $\hat T$, we obtain
\begin{align}\label{eqn:hatwest}
\|\hat \bw\|_{H^m(\hat T)}^2 \le 
C\Big(\sum_{\hat e\subset \p \hat T} \Big|\int_{\hat e} \hat \bw \cdot \hat \bt\Big|^2 
 + \mathop{\sup_{\hat q\in \hat Q_0}}_{\|\hat q\|_{L^2(\hat T)}=1} \Big| \int_{\hat T} (\hat{\rm rot}\,\hat \bw)\hat q\Big|^2\Big).
\end{align}
We now estimate the two terms on the right-hand side of \eqref{eqn:hatwest} separately.

First by a change of variables, we have
\begin{align*}
\int_{\hat e} \hat \bw\cdot   \hat \bt
& = \int_e \bw\cdot \bt =  \int_{e_R} {\bm f}\cdot \bt_{e_R} -  \int_{e} {\bm f} \cdot \bt.
\end{align*}
Set $\Theta:=G_h\circ G^{-1}$ so that $e = \Theta(e_R)$ and $T = \Theta(T_R)$.
There  holds  \cite[Proposition 3]{Lenoir86}
\begin{equation}\label{eqn:ThetaProperties}
|\Theta(x)-x| = \mathcal{O}(h^3_T),\quad|D\Theta-I_2| = \mathcal{O}(h^2_T),\quad \bt(\Theta(x)) = \frac{D\Theta \bt_{e_R}}{|D\Theta \bt_{e_R}|}(x),\qquad x\in \bar T_R,
\end{equation}
and therefore by a change of variables,
\begin{align*}
\int_e {\bm f}\cdot \bt  = 
\int_{e_R} |(D\Theta) \bt_{e_R}| ({\bm f}\cdot \bt)\circ \Theta  = \int_{e_R} ({\bm f} \circ \Theta) \cdot (D\Theta \bt_{e_R}).
\end{align*}
Thus,
\begin{align*}
\int_{\hat e} \hat \bw \cdot \hat \bt
& = \int_{e_R} \big({\bm f}\cdot \bt_{e_R} -  ({\bm f} \circ \Theta) \cdot (D\Theta \bt_{e_R})\big)\\
& = \int_{e_R} \big({\bm f} - ({\bm f}\circ \Theta)\big)\cdot \bt_{e_R} -  ({\bm f} \circ \Theta) \cdot (D\Theta \bt_{e_R} - \bt_{e_R})\big),
\end{align*}
and therefore by \eqref{eqn:ThetaProperties}, Taylor's Theorem, and a Sobolev embedding,
\begin{align}\label{eqn:whatt}
\Big|\int_{\hat e} \hat \bw \cdot \hat \bt\Big|
&\le C\big(h^4_T |{\bm f}|_{W^{1,\infty}(\bbR^2)}+h_T^3 \|{\bm f}\|_{L^\infty(\bbR^2)}\big)\le C h_T^3 \|{\bm f}\|_{H^3(\Omega)}.
\end{align}

Next wet let $\hat q\in \hat Q_0$ with $\|\hat q\|_{L^2(\hat T)} = 1$ and compute
\begin{align*}
 \int_{\hat T} (\hat{\rm rot}\,\hat \bw)\hat q
 & = \int_T {\rm rot}\, \bw q = \int_{T\cap T_R} ({\rm rot}\,{\bm f})q - \int_T ({\rm rot}\,{\bm f}) q = \int_{T\backslash T_R} ({\rm rot}\,{\bm f}) q.
 \end{align*}
 where $q\in Q_0(T)$ with $q(x) = \hat q(\hat x)$.  Using $\|q\|_{L^2(T)}\le C h_T \|\hat q\|_{L^2(\hat T)}\le C h_T$, we obtain
 \begin{align}\label{eqn:whatrot}
 \int_{\hat T} (\hat{\rm rot}\,\hat \bw)\hat q\le |T\backslash T_R| \|{\rm rot}\,{\bm f}\|_{L^{\infty}(\bbR^2)} \|q\|_{L^2(T)}
 \le C h_T^3  \|{\bm f}\|_{H^3(\Omega)} \|q\|_{L^2(T)}\le C h_T^4  \|{\bm f}\|_{H^3(\Omega)}.
 \end{align}

Applying estimates \eqref{eqn:whatt}--\eqref{eqn:whatrot} to \eqref{eqn:hatwest} yields
\begin{align*}
\|\hat \bw\|_{H^m(\hat T)}\le C h_T^3 \|{\bm f}\|_{H^{3}(\Omega)}.
\end{align*}
Therefore
\begin{align*}
\|\bPi_W {\bm f} - \bI_W {\bm f}\|_{L^2(T)} = \|\bw\|_{L^2(T)}\le C h_T \|R_T \hat \bw\|_{L^2(\hat T)}\le C \|\hat \bw\|_{L^2(\hat T)}\le C h_T^3  \|{\bm f}\|_{H^3(\Omega)}.
\end{align*}
Finally by \eqref{eqn:PiFirst} and the triangle inequality,
\[
\|{\bm f}- \bPi_W {\bm f}\|_{L^2(T)}\le C h_T^3 \|{\bm f}\|_{H^3(\Omega)}.
\]
Summing over $T\in \mct$ yields the estimate \eqref{eqn:PiWEst}:
\begin{align*}
\|{\bm f}-\bPi_W {\bm f}\|_{L^2(\Omega_h)}
&\le C \Big(\sum_{T\in \mct} h_T^6 \|{\bm f}\|_{H^3(\Omega)}^2\Big)^{1/2}\le 
C h^2 \|{\bm f}\|_{H^3(\Omega)} \Big(\sum_{T\in \mct} h_T^2\Big)^{1/2}\le C h^2 \|{\bm f}\|_{H^3(\Omega)}.
\end{align*}

\end{enumerate}

\section{Numerical Experiments}\label{sec:Numerics}

In this section we perform a simple set of numerical experiments
and compare the results with the theory established in the previous section.
We let $\Omega =B_1(0)\subset \bbR^2$ be the unit ball, and take the data such that the exact solution
is given by
\[
\bu = 
\begin{pmatrix}
(x_1^2+x_2^2-1)(8x_1^2 x_2+x_1^2+5x_2^2-1)\\
-4x_1(x_1^2+x_2^2-1)(3x_1^2+x_2^2+x_2-1)
\end{pmatrix},\qquad
p = 10(x_1^2+x_2^2-\frac12).
\]

We compute the finite element method \eqref{eqn:FEM}, taking the source approximation ${\bm f}_h$
to be the quadratic (nodal) Lagrange interpolant of ${\bm f}$, and the viscosity $\nu=10^{-1}$.  
The errors 
for a decreasing sequence of mesh parameters $h$ are depicted in Figure \ref{fig:Test1}--\ref{fig:Test1B}.
For comparison, we also plot the errors of the analogous Scott-Vogelius finite element method
using affine approximations, i.e., method \eqref{eqn:FEM} with $\bV^h\times Q^h$ replaced by $\tilde \bV^h\times \tilde Q^h$.  The Figure shows the asymptotic convergence rates
\[
\|\bu-\bu_h\|_{L^2(\Omega_h)}   = \mathcal{O}(h^3),\quad \|\nab (\bu-\bu_h)\|_{L^2(\Omega_h)} = \mathcal{O}(h^2),\quad
\|p-p_h\|_{L^2(\Omega_h)} = \mathcal{O}(h^2),
\]
for the isoparametric approximations.  These results agree with the theoretical results stated in Theorem \ref{thm:Main}.
In contrast, the numerics indicate  the solution of the affine approximation, denoted by $(\bu^{aff},p_h^{aff})\in \tilde \bV^h\times \tilde Q^h$ satisfies
the sub-optimal convergence rates
\[
\|\bu-\bu^{aff}_h\|_{L^2(\tilde \Omega_h)}   = \mathcal{O}(h^2),\quad \|\nab (\bu-\bu^{aff}_h)\|_{L^2(\tilde \Omega_h)} = \mathcal{O}(h^{3/2}),\quad
\|p-p^{aff}_h\|_{L^2(\tilde \Omega_h)} = \mathcal{O}(h^{3/2}).
\]

We also solve the finite element method \eqref{eqn:FEM} but with isoparametric spaces
defined via the usual composition, i.e., with velocity-pressure pair \eqref{eqn:pairComp}.
Numerical experiments indicate the method is stable 
and converges with optimal order.  However, as Figure \ref{fig:Test1B} shows,
the method is not divergence--free (nor pressure robust).

\begin{center}
\begin{figure}

\begin{tikzpicture}[scale=0.9]
\begin{loglogaxis}[xlabel={$h$},ylabel={},legend pos=south west,x dir=reverse]

\addplot coordinates {
(0.200,  7.35070e-02)
(0.100,  1.24639e-02)  
(0.050,  1.82328e-03) 
(0.025,  2.97070e-04)  
(0.013,  3.73755e-05)
}; 

\addplot coordinates {
(0.200,  8.88738e-02)
(0.100,  1.95354e-02)  
(0.050,  4.45288e-03)  
(0.025,  1.09608e-03)  
(0.013,  2.68463e-04)  
} ;

\legend{$\|\bu-\bu_h\|_{L^2(\Omega_h)}$,$\|\bu-\bu_h^{aff}\|_{L^2(\tilde \Omega_h)}$}
\end{loglogaxis}
\end{tikzpicture}
\begin{tikzpicture}[scale=0.9]
\begin{loglogaxis}[xlabel={$h$},ylabel={},legend pos=south west,x dir=reverse]

\addplot coordinates {
(0.200,  1.59679e+00)
(0.100,  5.06952e-01)  
(0.050,  1.47748e-01)  
(0.025,  4.26637e-02)  
(0.013, 1.05260e-02 ) 
};

\addplot coordinates {
(0.200,  1.50788e+00)
(0.100,  5.10514e-01)  
(0.050,  1.66926e-01)  
(0.025,  5.90583e-02)  
(0.013,  2.01655e-02)  
};  

\legend{$\|\nab (\bu-\bu_h)\|_{L^2(\Omega_h)}$,$\|\nab (\bu-\bu_h^{aff}\|_{L^2(\tilde\Omega_h)})$}
\end{loglogaxis}
\end{tikzpicture}

\caption{\label{fig:Test1}Velocity errors of the isoparametric Scott-Vogelius finite element method \eqref{eqn:FEM} (blue)
and the affine Scott-Vogelius method (red)
for decreasing values of mesh parameter $h$.}

\end{figure}
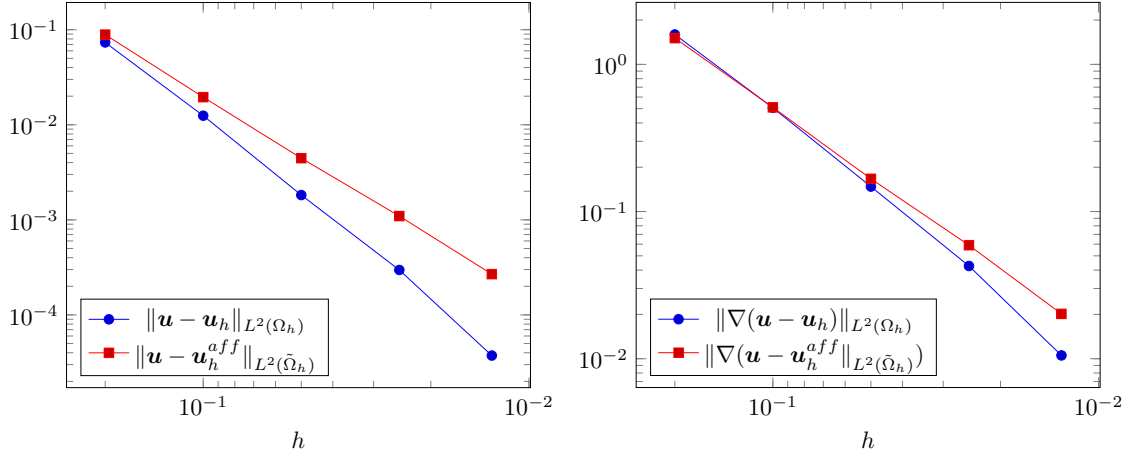
\end{center}

\begin{center}
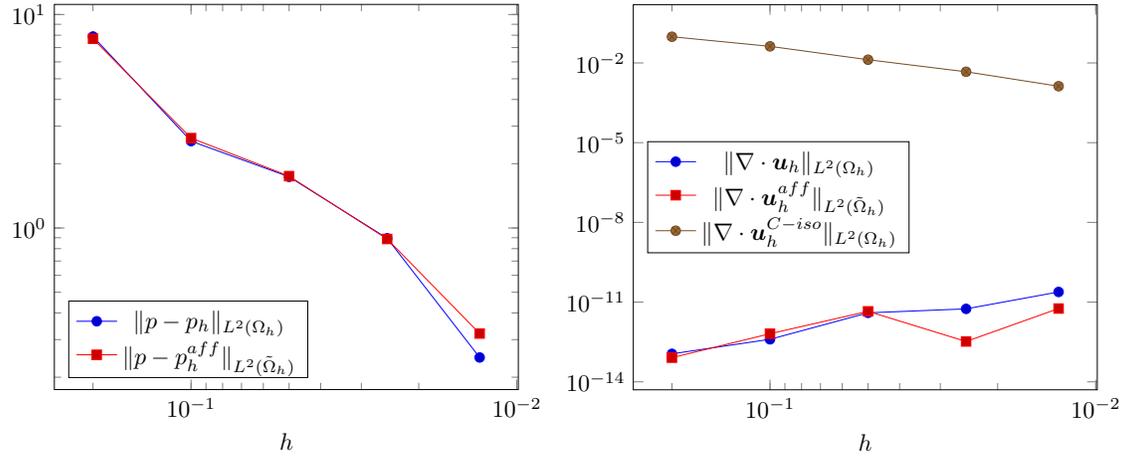
\begin{figure}
\begin{tikzpicture}[scale=0.9]
\begin{loglogaxis}[xlabel={$h$},ylabel={},legend pos=south west,x dir=reverse]

\addplot coordinates {
(0.200,  7.88712e+00)
(0.100,  2.55199e+00) 
(0.050,  1.73638e+00)  
(0.025,  8.94864e-01)  
(0.013,  2.47382e-01)  
};

\addplot coordinates {
(0.200,  7.69728e+00)
(0.100,  2.63697e+00)  
(0.050,  1.74916e+00)  
(0.025,  8.88013e-01) 
(0.013,  3.20275e-01)  
};

\legend{$\|p-p_h\|_{L^2(\Omega_h)}$,$\|p-p_h^{aff}\|_{L^2(\tilde \Omega_h)}$}
\end{loglogaxis}
\end{tikzpicture}
\begin{tikzpicture}[scale=0.9]
\begin{loglogaxis}[xlabel={$h$},ylabel={},x dir=reverse,legend style={at={(0.03,0.5)},anchor=west}]

\addplot coordinates {
(0.200,  1.12799e-13)
(0.100,  3.97904e-13) 
(0.050,  3.94818e-12) 
(0.025,  5.56621e-12) 
(0.013,  2.41007e-11) 
};

\addplot coordinates {
(0.200,  8.08242e-14)
(0.100,  6.45706e-13) 
(0.050,  4.48863e-12) 
(0.025,  3.29514e-13) 
(0.013,  5.72475e-12) 
};

\addplot coordinates{
(0.200,  9.65601e-02)
(0.100,  4.20591e-02)  
(0.050,  1.32019e-02)  
(0.025,  4.65026e-03)  
(0.013,  1.32950e-03) 
};
\legend{$\|\nab \cdot \bu_h\|_{L^2(\Omega_h)}$,$\|\nab \cdot \bu_h^{aff}\|_{L^2(\tilde \Omega_h)}$,$\|\nab \cdot \bu_h^{C-iso}\|_{L^2(\Omega_h)}$}
\end{loglogaxis}
\end{tikzpicture}

\caption{\label{fig:Test1B}Left: Pressure errors of the isoparametric Scott-Vogelius finite element method \eqref{eqn:FEM} (blue)
and the affine Scott-Vogelius method (red)
for decreasing values of mesh parameter $h$. Right: Divergence errors of the isoparametric Scott-Vogelius finite element method (blue),
the affine Scott-Vogelius method (red), and the isoparametric Scott-Vogelius using the standard composition of isoparametric mappings (brown).}
\end{figure}
\end{center}

\appendix

\section{Proofs of Preliminary Results}\label{sec:ProofPrelim}

\subsection{Proof of Lemma \ref{lem:ATLem}}
\begin{proof}
For notational simplicity, we set $\hat g(\hat x) = \det(DF_T(\hat x))$.
Using the fact that $DF_T\to \det(DF_T)$ is quadratic in two dimensions
and the estimates \eqref{eqn:FProperties}, a simple calculation shows
$|\hat g|_{W^{m,\infty}(\hat T)}\le C h^{2+m}_T$.  Consequently, by the quotient rule, 
for any multi-index $\alpha$ with $|\alpha|=m$,
\begin{align*}
\Big|\frac{\p^m}{\p \hat x^\alpha} \frac1{\hat g}\Big|
&\le C \sum_{|\beta^{(1)}|+|\beta^{(2)}|+\cdots +|\beta^{(m)}|=m} \frac{|\p^{|\beta^{(1)}|} \hat g/\p \hat x^{\beta^{(1)}}|\cdots |\p^{|\beta^{(m)}|} \hat g/\p \hat x^{\beta^{(m)}}|}{|\hat g^{m+1}|}\\
&\le C \sum_{|\beta^{(1)}|+|\beta^{(2)}|+\cdots +|\beta^{(m)}|=m} \frac{(h_T^{2+|\beta^{(1)}|})\cdots  (h_T^{2+|\beta^{(m)}|})}{|\hat g^{m+1}|}
\le C \frac{h_T^{3m}}{|\hat g^{m+1}|}\le C h_T^{m-2},
\end{align*}
where we used \eqref{eqn:FProperties} in the last inequality.

We then use the product rule and \eqref{eqn:FProperties} to find, for any $i,j\in \{1,2\}$
and multi-index $\alpha$ with $|\alpha|=m$,
\begin{align*}
\Big|\frac{\p^m (A_T)_{i,j}}{\p \hat x^{\alpha}}\Big| 
& =\Big| \frac{\p^m }{\p \hat x^\alpha} \big({(DF_T)_{i,j}}/{\hat g}\big)\Big|\\
&\le C \sum_{|\beta|+|\gamma|=m} \big| \p^\beta (DF_T)_{i,j}/\p^{|\beta|} \hat x\big| \big| \p^\gamma \hat g^{-1}/\p^{|\gamma|} \hat x\big|\\
&\le C \sum_{|\beta|+|\gamma|=m}  \big(h_T^{1+|\beta|}\big)\big(h_T^{|\gamma|-2}\big)\le C h_T^{m-1}.
\end{align*}
This establishes the first inequality in \eqref{eqn:ATbound}.

Next, we use the identity $A_T^{-1} = \det(DF_T)(DF_T)^{-1} = {\rm adj}(DF_T)$, the adjugate matrix of $DF_T$.
Because the entries of $DF_T$ and ${\rm adj}(DF_T)$ are the same up to permutation and sign in two dimensions,
we have by \eqref{eqn:FProperties},
\[
|A_T^{-1}|_{W^{m,\infty}(\hat T)}  = |DF_T|_{W^{m+1,\infty}(\hat T)}\le \left\{
\begin{array}{ll}
C h_T^{1+m} & m=0,1\\
0 & m\ge 2
\end{array}\right.
\]\hfill
\end{proof}

\subsection{Proof of Lemma \ref{lem:cofactorOnEdge}}

\begin{proof}
Let $\hat \bt$ be the unit tangent vector of $\hat e$ obtained by rotating $\hat \bn$ $90$ degrees clockwise.
Then a  calculation shows 
\[
\det(DF_T(\hat x))(DF_T(\hat x))^{-\intercal}\hat \bn  = \begin{pmatrix}
-(DF_T(\hat x) \hat \bt)_2\\
(DF_T(\hat x) \hat \bt)_1
\end{pmatrix}.
\]
Because $F_T$ restricted to $\hat e$ is affine, $(DF_T(\hat x) \hat \bt)$ is constant on $\hat e$.
This proves the lemma.
\end{proof}
\end{document}